\documentclass[10pt,a4paper]{article}

\usepackage[utf8]{inputenc}
\usepackage[T1]{fontenc}
\usepackage{lmodern}

\usepackage{graphicx}
\makeatletter
\def\maxwidth{\ifdim\Gin@nat@width>\linewidth\linewidth\else\Gin@nat@width\fi}
\makeatother
\setkeys{Gin}{width=\maxwidth,keepaspectratio}

\usepackage[margin=2.5cm]{geometry}
\usepackage{setspace}
\usepackage{enumitem}
\setlist{itemsep=0.25\baselineskip, parsep=0pt, topsep=0.25\baselineskip, partopsep=0pt, leftmargin=*}

\usepackage{tocloft}
\usepackage{amsmath,amssymb,mathtools,amsthm}
\usepackage{longtable}

\swapnumbers

\theoremstyle{plain}
\newtheorem{theorem}{Theorem}
\newtheorem{lemma}[theorem]{Lemma}
\newtheorem{proposition}[theorem]{Proposition}
\newtheorem{corollary}[theorem]{Corollary}

\theoremstyle{definition}
\newtheorem{definition}[theorem]{Definition}
\newtheorem{example}[theorem]{Example}
\newtheorem{remark}[theorem]{Remark}

\usepackage[numbers,sort&compress]{natbib}

\usepackage{url}
\usepackage{hyperref}
\hypersetup{hidelinks}
\usepackage{authblk}
\usepackage{orcidlink}

\newcommand{\IR}{\mathbb{R}} % Real numbers #listed
\newcommand{\IN}{\mathbb{N}} % Natural numbers #listed
\newcommand{\IZ}{\mathbb{Z}} % Integers #listed
\newcommand{\IQ}{\mathbb{Q}} % Rational numbers #listed
\newcommand{\ik}{\Bbbk} % Ground field #listed
\newcommand{\set}[1]{\{#1\}} % Simple set braces #listed
\newcommand{\KP}{\mathcal{P}}

\newcommand{\eps}{\varepsilon} % Epsilon variant #listed
\newcommand{\tightlist}{\setlength{\itemsep}{0pt}\setlength{\parskip}{0pt}}
\renewcommand{\part}{\vdash} % Turnstile #listed (overrides LaTeX \part sectioning)
\newcommand{\KM}{\mathcal{M}} % Multiset / multi-index set #listed
\newcommand{\Part}{\mathrm{Part}} % Partitions of a set #listed
\newcommand{\Cov}{\mathrm{Cov}} % Coverings #listed
\DeclareMathOperator{\lf}{\mathrm{leaf}} % Leaf support #listed
\DeclareMathOperator{\supp}{\mathrm{Supp}} % Support #listed
\title{Discrete Faà di Bruno via Möbius Inversion}
\author{Heinrich Hartmann\footnote{Hartmann IT GmbH / \href{https://heinrichhartmann.com/math}{heinrichhartmann.com/math} / \texttt{Heinrich@HeinrichHartmann.com}}\;\,\orcidlink{0000-0002-3929-2421}}
\date{}

\begin{document}

\maketitle
\vspace{-3em}\begin{center}\href{https://heinrichhartmann.com/math/2026-Moebius-Faa-di-Bruno}{heinrichhartmann.com/math/2026-Moebius-Faa-di-Bruno}\end{center}\vspace{0.5em}

\begin{abstract}
We approach discrete and differential Faà di Bruno formulas from a Möbius inversion angle. 
On the Boolean cube, Newton's discrete Taylor formula and the definition of iterated forward differences form a zeta--Möbius dual pair, 
and composing two Taylor expansions and inverting once yields a closed discrete 
Faà di Bruno formula at a fixed basepoint: for arbitrary maps $f, g$ between abelian groups,

$$
\Delta(f \circ g;\,x;\,u_1,\dots,u_k) = \sum_{H \in \mathrm{Cov}(k)} \Delta(f;\,g(x);\,(\Delta(g;x;u_T))_{T\in H}),
$$

where $\mathrm{Cov}(k)$ denotes the coverings of $[k]$ by nonempty subsets.
Grouping repeated directions gives binomial versions on multi-index grids, and iterating gives formulas for $m$-fold composites, 
with integer covering coefficients governed by explicit cross and level recursions, a discrete analogue of the Constantine--Savits formulas.

The relationship between coverings and partitions appearing in classical Faà di Bruno formulas, is exhibited in an algebraic setting. 
The discrete formulas are Taylor expansions over the function algebra of the Boolean cube, 
$B_k = \mathbf{k}[\delta_1,\dots,\delta_k]/(\delta_i^2-\delta_i)$, whose idempotent generators absorb overlapping products; 
in the differential analogue $A_k = \mathbf{k}[\varepsilon_1,\dots,\varepsilon_k]/(\varepsilon_i^2)$, nilpotent generators annihilate overlaps and only partitions remain. 
Both algebras are fibers of the flat deformation $C_k = \mathbf{k}[t][x_1,\dots,x_k]/(x_i^2 - tx_i)$, over which a single weighted covering formula interpolates: 
its coefficients are difference quotients, non-partition coverings carry positive powers of $t$, and evaluation at $t = 0$ yields 
the classical partition-indexed Faà di Bruno formula.

We demonstrate how these algebraic identities can be lifted to the analytical setting of $C^n$ maps between Banach spaces,
recovering the multivariate Faà di Bruno formula of Constantine--Savits and extending it to composites of several maps. 
Boolean finite differences, binomial grid formulas, infinitesimal Taylor algebras, and Fréchet derivatives thus appear 
as four realizations of one Möbius-dual Faà di Bruno formula, connected by a flat family.
\end{abstract}

\tableofcontents

\section{Introduction}\label{introduction}

We approach discrete and differential calculus from the point of view of
Möbius inversion. The motivating observation is elementary: the discrete
Taylor formula is the Möbius inverse of the definition of iterated
forward differences. If \(g: X \to Y\) is a map between abelian groups,
then the values \(g(x + \sum_{i \in S} u_i)\) for \(S \subseteq [k]\)
form a Boolean cube, and the forward difference

\[
  \Delta(g; x; u_1, \dots, u_k) = \sum_{S \subseteq [k]} (-1)^{k-|S|}\, g(x + {\textstyle\sum_{i \in S}} u_i)
\]

is exactly the Möbius transform of this cube. Newton's finite Taylor
formula is the inverse zeta expansion.

\begin{remark}[Notation]

We prefer explicit arguments over subscript notation, writing
\(\Delta(g; x; u)\) instead of \(\Delta_u g(x)\). Semicolons and commas
are both argument separators; we use semicolons as a visual hint when
arguments are of different kinds. Dropping trailing arguments denotes
currying: \(\Delta(g; x)\) is the map \(u \mapsto \Delta(g; x; u)\).

\end{remark}

This observation extends without changing its nature from Boolean cubes
to multi-index grids. Repeating directions replaces subsets by slot sets
and gives a binomial form of Möbius inversion. Thus the discrete Taylor
formula and its higher multi-index versions are two instances of the
same zeta/Möbius duality.

The same principle carries over to the differential setting once one
works over an infinitesimal Taylor algebra. The discrete calculus is
algebra over the Boolean cube algebra
\(B_k = \ik[\delta_1, \dots, \delta_k]/(\delta_i^2 - \delta_i)\), whose
spectrum is the cube \(\set{0,1}^k\), or over the grid algebra
\(B_k^\nu = \ik[\delta_1, \dots, \delta_k]/((\delta_i)_{\nu_i+1})\) with
falling-factorial relations, whose natural module basis is the Newton
basis \((\delta_i)_j\). The infinitesimal analogues replace idempotent
relations by nilpotent ones: the Taylor algebra
\(A_k = \ik[\eps_1, \dots, \eps_k]/(\eps_i^2)\) and its truncation
\(A_k^\nu = \ik[\eps_1, \dots, \eps_k]/(\eps_i^{\nu_i+1})\). We will see
below that these are naturally expressed as two fibers of the same flat
deformation.

Working in \(A_k^\nu\), the Möbius duality transfers exactly from finite
grids to infinitesimal grids. Evaluating a polynomial map on the
infinitesimal grid \(x + \sum_i \beta_i \eps_i v_i\) for
\(0 \leq \beta \leq \nu\) gives a Taylor expansion on the zeta side, and
the same alternating sieve extracts the differential monomials: in
\(A_k^\gamma\),

\[
  D^\gamma(p; x; v_\bullet)\, \eps^\gamma = \sum_{\beta \leq \gamma} (-1)^{|\gamma - \beta|} \binom{\gamma}{\beta}\, p(x + {\textstyle\sum_i} \beta_i \eps_i v_i).
\]

In this sense, the finite-difference and differential Taylor formulas
are parallel Möbius-dual constructions: finite grids give forward
differences, infinitesimal grids give differentials.

This low-tech viewpoint becomes especially useful for composition. Our
first result is a discrete Faà di Bruno formula indexed by coverings,
closed at the original basepoint.

\begin{theorem}[Discrete Faà di Bruno]

Let \(X, Y, Z\) be abelian groups, \(g: X \to Y\), \(f: Y \to Z\)
arbitrary maps, \(x \in X\), and \(u_1, \dots, u_k \in X\). Then

\[
  \Delta(f \circ g; x; u_1, \dots, u_k) = \sum_{H \in \Cov(k)} \Delta(f; g(x); (\Delta(g; x; u_T))_{T \in H}),
\]

where \(\Cov(k)\) denotes the set of coverings of \([k]\) by nonempty
subsets.

\end{theorem}

The formula is exact and integral, closed at the original basepoint
\(g(x)\), and uses only the original direction increments
\(\Delta(g; x; u_T)\). It follows by composing two discrete Taylor
expansions and applying Boolean Möbius inversion: the zeta side of the
composite is trivially the composite of the zeta sides, and the covering
formula is its Möbius inverse.

For \(k = 2\) there are five coverings of \(\set{1,2}\). Writing
\(y = g(x)\) and \(g_T = \Delta(g; x; u_T)\),

\[
  \Delta(f \circ g; x; u_1, u_2)
  = \underbrace{\Delta(f; y; g_1, g_2) + \Delta(f; y; g_{12})}_{\text{partitions, weight } 2}
  + \underbrace{\Delta(f; y; g_1, g_{12}) + \Delta(f; y; g_2, g_{12})}_{\text{weight } 3}
  + \underbrace{\Delta(f; y; g_1, g_2, g_{12})}_{\text{weight } 4},
\]

where the \emph{weight} of a covering is the total size
\(\mathrm{wt}(H) = \sum_{T \in H} |T|\) of its blocks. Under the scaling
\(u_i = t v_i\), a term of weight \(w\) is of order \(t^w\): dividing by
\(t^2\) and letting \(t \to 0\), the three overlapping coverings vanish
and the two partition terms become the classical second-order chain rule
\(Df \cdot D^2g(v_1, v_2) + D^2f(Dg\, v_1, Dg\, v_2)\). In general the
covering counts \(1, 5, 109, 32297, \dots\)
\href{https://oeis.org/A003465}{A003465} replace the Bell numbers
\(1, 2, 5, 15, \dots\) \href{https://oeis.org/A000110}{A000110} of the
smooth theory.

Traditionally, this identity is grouped differently. In the example
above, every covering except the partition \(\set{\set{1},\set{2}}\)
contains the block \(\set{1,2}\); the corresponding four terms
constitute the finite Taylor expansion of a single difference of \(f\)
at the shifted basepoint \(y + g_1 + g_2\), and are commonly contracted
into it:

\[
  \Delta(f \circ g;\, x;\, u_1, u_2)
  = \underbrace{\Delta(f;\, y;\, g_1, g_2)}_{\set{\set{1},\set{2}}}
  + \underbrace{\Delta(f;\, y + g_1 + g_2;\, g_{12})}_{\set{\set{1,2}}},
\]

two terms instead of five, indexed by the partitions of \(\set{1,2}\).
Duarte and Torres \citep{DT2008} show that such a contraction is always
possible: the finite-difference chain rule can be written as a
partition-indexed sum with recursively shifted basepoints and corrected
directions, and this is the direction in which the theory has developed.
The contracted formulas are compact and of high practical utility, but
they are not closed forms, and the basepoint shifts break the symmetry
in the direction arguments, which makes higher-composition questions for
\(f_m \circ \cdots \circ f_1\) very difficult to approach. We argue that
the expanded fixed-basepoint covering form is the natural way to view
the formula: it is closed and symmetric in the directions, it carries
the weight grading that governs the collapse to the classical chain
rule, and it generalizes directly.

The same construction iterates through composites of several maps and
refines to multi-index grids, and both coefficient systems of the
resulting duality are explicit.

\begin{theorem}[Iterated Faà di Bruno duality]

Let
\(X_0 \xrightarrow{f_1} X_1 \xrightarrow{f_2} \cdots \xrightarrow{f_m} X_m\)
be arbitrary maps between abelian groups, \(x \in X_0\), directions
\(u_s \in X_0\) for \(s \in S\), \(\gamma \in \IN_0^S\), and
\(z = (f_m \circ \cdots \circ f_1)(x)\). Then

\[
  \Delta(f_m \circ \cdots \circ f_1;\, x;\, u_\bullet) = \sum_{K \in \Cov_m(S)} \Delta^{K}(f_1, \dots, f_m;\, x;\, u_\bullet),
\]

\[
  \Delta(f_m \circ \cdots \circ f_1; x; u_\bullet^\gamma) = \sum_{\kappa \in \KM_+^m(S)} \Cov_m(\gamma, \kappa)\, \Delta^{\kappa}(f_1, \dots, f_m; x; u_\bullet),
\]

\[
  (f_m \circ \cdots \circ f_1)(x + {\textstyle\sum_s} \gamma_s u_s) = z + \sum_{\kappa \in \KM_+^m(S)} \mathrm{Pow}_m(\gamma, \kappa)\, \Delta^{\kappa}(f_1, \dots, f_m; x; u_\bullet),
\]

where \(\Cov_m(S)\) is the set of \(m\)-fold iterated coverings of
\(S\), the binomial sums range over the \(m\)-fold iterated multi-index
sets \(\KM_+^m(S) = \KM_+(\KM_+^{m-1}(S))\), the iterated increments
\(\Delta^{K}\) and \(\Delta^{\kappa}\) are formed by nesting inner
differences as directions of outer differences, and
\(\mathrm{Pow}_m(\gamma, \kappa)\) and \(\Cov_m(\gamma, \kappa)\) count
\(m\)-fold iterated subsets, respectively coverings, of the slot set
\(S(\gamma)\) with profile \(\kappa\). The two coefficient systems are
Möbius inverses of each other in \(\gamma\),

\[
  \Cov_m(\gamma, \kappa) = \sum_{\beta \leq \gamma} (-1)^{\mathrm{wt}(\gamma - \beta)} \binom{\gamma}{\beta}\, \mathrm{Pow}_m(\beta, \kappa),
\]

and both satisfy explicit level recursions in \(m\).

\end{theorem}

The coefficients are nonnegative integers, and no factorial denominators
appear anywhere in the discrete theory: all identities hold over
arbitrary abelian groups, in any characteristic.

On the differential side, the same Taylor-composition argument produces
the classical partition-indexed Faà di Bruno formulas. The reason
coverings become partitions is visible already in the multiplication
laws of the cube algebras. Idempotent generators absorb overlaps,
nilpotent generators annihilate them:

\[
  \delta^T \delta^U = \delta^{T \cup U} \quad \text{in } B_k,
  \qquad
  \eps^T \eps^U = \eps^{T \cup U} \cdot [T \cap U = \emptyset]
  \quad \text{in } A_k.
\]

A product of faces contributes to the top face exactly when the
exponents cover \([k]\); in \(A_k\) only the disjoint coverings
contribute. Covering logic becomes partition logic.

The two multiplication laws are fibers of one family. Over \(\ik[t]\),
the deformation cube algebra

\[
  C_k = \ik[t][x_1, \dots, x_k]/(x_i^2 - t x_i), \qquad x^S x^T = t^{|S \cap T|}\, x^{S \cup T},
\]

is free as a \(\ik[t]\)-module; its fiber at \(t = 0\) is the Taylor
algebra \(A_k\), and after inverting \(t\) the rescaling
\(x_i = t \delta_i\) identifies it with the Boolean cube algebra.
Expanding \(p(x + \sum_i v_i x_i)\) in the basis \(x^S\) produces
coefficients \(c_S(t) \in \ik[t]\) that interpolate the two calculi:

\[
  t^{|S|}\, c_S(t) = \Delta(p; x; (t v_i)_{i \in S}), \qquad c_S(0) = D(p; x; v_S).
\]

The two readings of \(c_S(t)\) explain each other. Since \(C_k\) is free
over \(\ik[t]\) with basis \(x^S\), the coefficients \(c_S(t)\) are
polynomials in \(t\) from the outset. Evaluating the expansion at the
vertices of the rescaled cube \(x_i = t \delta_i\),
\(\delta \in \set{0,1}^k\), and applying Boolean Möbius inversion
identifies \(t^{|S|} c_S(t)\) with the forward difference
\(\Delta(p; x; (t v_i)_{i \in S})\). Combining the two statements: the
forward difference is divisible by \(t^{|S|}\) in \(\ik[t]\), the
difference quotient \(c_S(t)\) is itself a polynomial in \(t\), and its
value at \(t = 0\), computed in the fiber \(A_k\), is the derivative.
The passage from difference quotients to derivatives is thus a
divisibility statement about polynomials in \(t\), rather than an
analytic limit. Composing two such expansions yields a single Faà di
Bruno identity over \(\ik[t]\), indexed by coverings and weighted by
overlap:

\[
  [x^{[k]}]\, (f \circ g)(x + {\textstyle\sum_i} v_i x_i) = \sum_{H \in \Cov(k)} t^{\,\mathrm{wt}(H) - k}\; \mathrm{FdB}_H(t),
\]

where \(\mathrm{FdB}_H(t)\) is the covering term formed from difference
quotients. At \(t = 1\) this is the discrete covering formula; at
\(t = 0\) every overlapping covering vanishes and the classical
partition formula remains, as in the \(k = 2\) example above. The
covering and partition Faà di Bruno formulas are one identity, evaluated
at two fibers. The grid algebras deform in the same way, via
\(C_k^\nu = \ik[t][x_1, \dots, x_k]/(\prod_{j=0}^{\nu_i}(x_i - jt) : 1 \leq i \leq k)\)
with its Newton basis: in this family, Newton interpolation degenerates
literally to Taylor expansion.

For polynomial maps \(q: \ik^e \to \ik^d\), \(p: \ik^d \to \ik\), the
special fiber gives the classical partition form \citep{Levy2006}

\[
  D(p \circ q; x; v_{[k]}) = \sum_{\pi \in \Part([k])} D(p; q(x); (D(q; x; v_B))_{B \in \pi}).
\]

Grouping repeated directions gives the Constantine--Savits multi-index
formula \citep{CS1996}. Iterating the construction gives an iterated
Constantine--Savits formula for \(m\)-fold composites with recursive
partition coefficients \(\Part_m(\gamma, \kappa)\).

These polynomial identities lift to the Banach setting. For \(C^n\) maps
between Banach spaces, the Taylor jet of a composite agrees up to order
\(n\) with the composite of the individual Taylor jets:

\[
  T^n(f_m \circ \cdots \circ f_1;\, x) = \pi_{\leq n}(T^n(f_m;\, x_{m-1}) \circ \cdots \circ T^n(f_1;\, x_0)).
\]

Since the partition Faà di Bruno formulas are polynomial coefficient
identities, they apply directly to these Taylor jets and therefore to
any \(C^n\) composite. In the polynomial setting, the deformation is
precisely the passage to the limit of difference quotients: scaling the
increments by \(t\), the non-partition terms of the discrete covering
formula carry positive powers of \(t\) and vanish as \(t \to 0\), while
the partition terms specialize to their smooth counterparts
\ref{thm:fdb-deformation}.

Putting things together, we obtain the following generalization of
Constantine--Savits to Banach spaces and iterated multi-indices.

\begin{theorem}[Fréchet iterated Faà di Bruno / recursive Constantine--Savits]

Let \(X_0, \dots, X_m\) be Banach spaces, let \(f_r: X_{r-1} \to X_r\)
be \(C^n\) near \(x_{r-1}\), and put \(x_0 = x\),
\(x_r = f_r(x_{r-1})\). Fix directions \(v_1, \dots, v_k \in X_0\) and
let \(\gamma \in \IN_0^k\) with \(1 \leq |\gamma| \leq n\). Then

\[
  D^\gamma(f_m \circ \cdots \circ f_1;\, x;\, v_\bullet) = \sum_{\kappa \in \KM_+^m(k)} \Part_m(\gamma, \kappa)\, D^\kappa(f_1, \dots, f_m;\, x;\, v_\bullet),
\]

where \(\Part_m(\gamma, \kappa)\) counts \(m\)-fold partitions of the
slot set \(S(\gamma)\) with profile \(\kappa\). Writing
\(\lf(\lambda) \in \IN_0^k\) for the leaf multi-index of \(\lambda\),
obtained by recursively summing the leaves of \(\lambda\) with
multiplicity, the coefficients vanish unless
\(\sum_\lambda \kappa(\lambda)\, \lf(\lambda) = \gamma\), and satisfy
the closed recursion
\(\Part_1(\gamma, \alpha) = \delta_{\gamma, \alpha}\) and, for
\(m \geq 2\),

\[
  \Part_m(\gamma, \kappa) = \frac{\gamma!}{\kappa!}\, \prod_{\substack{\beta \in \KM_+^{m-1}(S) \\ \alpha = \lf(\beta)}} \left( \frac{1}{\alpha!} \cdot \Part_{m-1}(\alpha, \beta) \right)^{\kappa(\beta)}.
\]

For \(m = 2\) this is the Constantine--Savits coefficient
\(\Part_2(\gamma, \kappa) = \gamma! / (\kappa!\, \prod_\beta (\beta!)^{\kappa(\beta)})\).

\end{theorem}

A final section collects five short applications of the discrete
calculus, among them discrete and Fréchet product rules, governed by
ordered coverings and ordered partitions respectively, and an easy proof
that polynomial maps between abelian groups are closed under
composition, with \(\deg(f \circ g) \leq \deg f \cdot \deg g\). The
paper is accompanied by a software package, available on
\href{https://github.com/HeinrichHartmann/Discrete-Faa-di-Bruno}{GitHub},
containing symbolic expansions and validations of the covering formula
(degree \(4\), ca. \(2 \cdot 10^6\) terms) and of the Duarte--Torres
recursion (degree \(10\), ca. \(115{,}000\) terms).

Taken together: Boolean finite differences, binomial grid formulas,
infinitesimal Taylor algebras, and Fréchet derivatives are four fibers
of one Möbius-dual Taylor calculus.

\subsection{Related work}\label{related-work}

The closest discrete predecessor is the formula of Duarte and Torres
\citep{DT2008}, and the present work grew out of an attempt to
understand their dense and insightful paper properly, and to repackage
its results in a more functorial form. Their discrete Faà di Bruno
formula is indexed by partitions and uses recursively shifted basepoints
and corrected directions; the covering formula in this paper writes all
terms at the original basepoint \(g(x)\) and is indexed by coverings of
the direction set. The covering formula is fully expanded at the
original basepoint \(g(x)\); the Duarte--Torres formula is best
understood as a grouping of the same terms by shifted basepoint,
collecting all covering contributions that share a common recursively
shifted evaluation point into a single partition-indexed term. Thus the
two approaches describe the same finite-difference chain rule in
different coordinates: Duarte--Torres gives a compact recursive
partition formula with shifted basepoints and corrected directions,
whereas the present paper gives a closed fixed-basepoint covering
formula.

A second neighboring literature is the divided-difference form of Faà di
Bruno. Floater and Lyche \citep{FL2007} derive chain rules for divided
differences \([t_0, \dots, t_n]\) of univariate functions, organized by
the linear order of interpolation nodes; related divided-difference
forms were obtained by Wang and Wang \citep{WW2006} and extended by Xu
and Wang \citep{XW2010}. Our finite-difference formula is cubical rather
than linearly ordered: it concerns \(\Delta(g; x; u_1, \dots, u_k)\) at
one basepoint with independent directions, and its natural combinatorics
is that of coverings of \([k]\). In the univariate equal-step case,
normalized forward differences are divided differences on a uniform
grid, so the two settings meet after specialization. The deformation
cube makes this connection explicit: the coefficients \(c_S(t)\) of the
family \(C_k\) are precisely difference quotients that extend through
\(t = 0\) to derivatives. A general calculus built directly on such
scaled difference-quotient maps has been developed by Bertram, Glöckner,
and Neeb \citep{BGN2004}; our flat family can be read as a polynomial,
higher-order companion to their first-order framework.

On the smooth side, the multi-index Faà di Bruno formula of Constantine
and Savits \citep{CS1996} is one of the main targets recovered here. Our
infinitesimal Taylor algebra \(A_k^\nu\) gives a direct Möbius-dual
derivation of their coefficients from the square-free partition formula
by grouping repeated directions; closely related bookkeeping between
multi-indices and partitions of multisets appears in Hardy
\citep{Hardy2006}. The slot-set construction \(S(\gamma)\) provides a
uniform way to pass between partition formulas and multi-index formulas.
Iterated one-dimensional Faà di Bruno formulas have been studied through
higher-order Bell polynomials by Natalini and Ricci \citep{NR2004}; the
iterated formula in this paper gives a multivariate version for
\(m\)-fold composites with recursively defined partition coefficients.

The finite-difference operators used here also belong to the older
theory of cross-effects and polynomial maps. The alternating expression
\(\Delta(g; x; u_1, \dots, u_k)\) is the \(k\)-fold deviation or
cross-effect of a map between abelian groups, in the sense of Eilenberg
and Mac Lane \citep{EML1954}. Vanishing of sufficiently high differences
is the classical finite-difference characterization of polynomial
operations, going back to Mazur and Orlicz \citep{MO1934} and appearing
in later work such as Leibman's theory of polynomial mappings of groups
\citep{Leibman2002}; in the different setting of functors between
abelian categories, Bauer, Johnson, Osborne, Riehl, and Tebbe
\citep{BJORT2018} prove a categorified higher-order chain rule for
cross-effects, where equalities are replaced by chain homotopy
equivalences and similar combinatorics governs the resulting Faà di
Bruno decomposition. The present paper works entirely at the level of
maps between abelian groups, where the covering Faà di Bruno formula is
an explicit, exact chain rule for cross-effects; as a corollary,
polynomial maps are closed under composition with
\(\deg(f \circ g) \leq \deg f \cdot \deg g\) \ref{cor:degree-bound}.

Finally, Faà di Bruno formulas have a rich algebraic life in coalgebras,
Hopf algebras, and incidence algebras, beginning with work of Joni and
Rota \citep{JR1979} and surveyed from several viewpoints by Frabetti and
Manchon \citep{FM2014}. The present paper does not develop the
Hopf-algebraic formalism. Instead, it uses ordinary Möbius inversion on
Boolean and multi-index grids, together with the cube algebras \(B_k\),
\(A_k\), and their deformation \(C_k\).

\section{Möbius inversion}\label{muxf6bius-inversion}

Möbius inversion generalizes the elementary principle that individual
terms of a sum can be extracted via alternating differences. In its
general form, it applies to any locally finite partially ordered set
(see \citep{Stanley2012} for the general theory). In this text we need
it for Boolean lattices of subsets, and for the componentwise partial
order on multi-indices.

We write \([n] = \set{1, \dots, n}\), \(\KP(S)\) for the power set of a
finite set \(S\), \(\KP_+(S) = \KP(S) \setminus \set{\emptyset}\) for
the set of nonempty subsets, and \(\KP(n)\) for \(\KP([n])\). Note
\(|\KP(k)| = 2^k\). We think of maps \(a: \KP(S) \to X\) as \emph{cubes}
in \(X\) with \(2^{|S|}\) vertices \(a(T)\) for \(T \subseteq S\) and
legs \(a(\set{i})\).

\begin{proposition}[Boolean Möbius inversion]

\label{prop:moebius-inversion} Let \(G\) be an abelian group. For a cube
\(a: \KP(k) \to G\), set

\[
  \zeta(a; S) := \sum_{T \subseteq S} a(T), \qquad \mu(a; S) := \sum_{T \subseteq S} (-1)^{|S|-|T|}\, a(T).
\]

Then \(\zeta\) and \(\mu\) define inverse bijections on
\(\mathrm{Map}(\KP(k), G)\).

\end{proposition}

\begin{proof}

Exchange summation order and use
\((1-1)^{|S|} = \sum_{T \subseteq S} (-1)^{|S|-|T|} = \delta_{S,\emptyset}\).

\end{proof}

\begin{definition}[Multi-indices and slot sets]

\label{def:multi-index}

\begin{itemize}
\tightlist
\item
  A \emph{multi-index} on a finite set \(S\) is a map
  \(\alpha: S \to \IN_0\), i.e.~an element of \(\IN_0^S\). The
  componentwise partial order is \(\beta \leq \alpha\) iff
  \(\beta_s \leq \alpha_s\) for all \(s\).
\item
  The \emph{weight} is
  \(|\alpha| = \mathrm{wt}(\alpha) = \sum_{s \in S} \alpha_s\), the
  \emph{height} is \(\mathrm{ht}(\alpha) = \max_{s \in S} \alpha_s\).
\item
  The \emph{factorial} is \(\alpha! = \prod_{s \in S} \alpha_s!\), the
  \emph{falling factorial} is
  \((\alpha)_\beta = \prod_{s \in S} (\alpha_s)_{\beta_s}\) where
  \((n)_r = n(n-1)\cdots(n-r+1)\).
\item
  A \emph{Boolean realization} of \(\alpha \in \IN_0^S\) is a set
  \(S(\alpha) = \set{(s, i) : s \in S,\, 0 < i \leq \alpha_s}\) with
  canonical projection \(\pi: S(\alpha) \to S\).
\item
  The \emph{fiber measure} of a map \(q: S' \to S\) of finite sets is
  \(\nu(q) \in \IN_0^S\), \(\nu(q)_s = |q^{-1}(s)|\). In particular
  \(\nu(\pi) = \alpha\) and \(|S'| = \mathrm{wt}(\nu(q))\).
\item
  We embed \(\KP(S) \hookrightarrow \IN_0^S\) by \(U \mapsto 1_U\); the
  image consists of multi-indices of height \(\leq 1\), and
  \(S(1_U) \cong U\).
\end{itemize}

\end{definition}

We think of maps \(A: \IN_0^S \to X\) as \emph{grids} in \(X\). The
points \(A(1_T)\) for \(T \subseteq S\) form the coordinate cube; the
remaining points \(A(\alpha)\) for \(\mathrm{ht}(\alpha) > 1\) fill the
interior of the grid.

We define binomial Möbius inversion via the Boolean realization of
multi-indices.

\begin{proposition}[Binomial Möbius inversion]

\label{prop:multi-index-moebius} Let \(G\) be an abelian group. For a
map \(A: \IN_0^S \to G\) and \(\alpha \in \IN_0^S\), let
\(A_\alpha = A \circ \nu: \KP(S(\alpha)) \xrightarrow{\nu} \IN_0^S \xrightarrow{A} G\).
Set

\[
  \zeta(A; \alpha) := \zeta(A_\alpha; S(\alpha)) = \sum_{\beta \leq \alpha} \frac{(\alpha)_\beta}{\beta!}\, A(\beta), \qquad \mu(A; \alpha) := \mu(A_\alpha; S(\alpha)) = \sum_{\beta \leq \alpha} (-1)^{\mathrm{wt}(\alpha - \beta)}\, \frac{(\alpha)_\beta}{\beta!}\, A(\beta).
\]

Then \(\zeta\) and \(\mu\) define inverse bijections on
\(\mathrm{Map}(\IN_0^S, G)\), extending the Boolean case.

\end{proposition}

\begin{proof}

The number of subsets \(T \subseteq S(\alpha)\) with \(\nu(T) = \beta\)
is \(\frac{(\alpha)_\beta}{\beta!}\), giving the explicit sums. For the
inversion, observe that the cube \(A_\alpha = A \circ \nu\) satisfies
\((\zeta A)_\alpha = \zeta(A_\alpha)\): for \(T \subseteq S(\alpha)\)
with \(\nu(T) = \beta\), the restriction of \(A_\alpha\) to \(\KP(T)\)
is the realization of \(A_\beta\) on \(T\), hence
\(\zeta(A_\alpha; T) = \sum_{R \subseteq T} A(\nu(R)) = \zeta(A; \beta)\).
Applying Boolean Möbius inversion \ref{prop:moebius-inversion} to
\(A_\alpha\) on \(\KP(S(\alpha))\):
\(\mu(\zeta A; \alpha) = \mu((\zeta A)_\alpha; S(\alpha)) = \mu(\zeta(A_\alpha); S(\alpha)) = A_\alpha(S(\alpha)) = A(\alpha)\),
and symmetrically for \(\zeta \mu\). Alternatively, exchange summation
order, use
\(\binom{\alpha}{\beta}\binom{\beta}{\gamma} = \binom{\alpha}{\gamma}\binom{\alpha-\gamma}{\beta-\gamma}\)
to decouple, and apply the multi-index binomial theorem
\((1-1)^\alpha = \sum_{\beta \leq \alpha} (-1)^{\mathrm{wt}(\beta)} \binom{\alpha}{\beta} = \delta_{\alpha,0}\).

\end{proof}

\section{Discrete Möbius calculus}\label{discrete-muxf6bius-calculus}

This section develops the composition calculus of forward differences
for arbitrary maps between abelian groups: Taylor duality, the covering
Faà di Bruno formula, and their iterated and binomial refinements with
the coefficient systems \(\mathrm{Pow}_m\) and \(\Cov_m\). All
identities are exact with integer coefficients. Every formula evaluates
maps only at the points \(x + \sum_{i \in S} u_i\), so the domain may in
fact be any commutative monoid; we state results for abelian groups and
use monoid domains such as \(\IN_0\) in the applications.

\subsection{Taylor duality}\label{taylor-duality}

\begin{definition}[Forward differences and translations]

\label{def:forward-diff} Let \(X, Y\) be abelian groups, \(g: X \to Y\),
\(x \in X\), and directions \(u_\bullet = (u_1, \dots, u_k) \in X^k\).
The \emph{iterated forward difference} and the \emph{iterated
translation} are:

\[
  \Delta(g; x; u_\bullet) := \sum_{S \subseteq [k]} (-1)^{k-|S|}\, g(x + {\textstyle\sum_{i \in S}} u_i), \qquad T(g; x; u_\bullet) := g(x + {\textstyle\sum_{i=1}^k} u_i).
\]

For \(S \subseteq [k]\), write \(u_S = (u_i)_{i \in S}\). As \(\Delta\)
is invariant under permutation of directions, we can define:

\[
  \Delta^S(g; x; u_\bullet) := \Delta(g; x; u_S), \qquad T^S(g; x; u_\bullet) := g(x + {\textstyle\sum_{i \in S}} u_i).
\]

For \(\alpha \in \IN_0^k\), write
\(u_\bullet^\alpha = (u_1^{\times \alpha_1}, \dots, u_k^{\times \alpha_k})\)
for the direction list with \(u_i\) repeated \(\alpha_i\) times. Then:

\[
  \Delta^\alpha(g; x; u_\bullet) := \Delta(g; x; u_\bullet^\alpha) = \sum_{\beta \leq \alpha} (-1)^{|\alpha|-|\beta|}\, \frac{(\alpha)_\beta}{\beta!}\, g(x + {\textstyle\sum_i} \beta_i u_i), \qquad T^\alpha(g; x; u_\bullet) := g(x + {\textstyle\sum_{i=1}^k} \alpha_i u_i),
\]

where we grouped the \(2^{|\alpha|}\) Boolean terms of
\(\Delta(g; x; u_\bullet^\alpha)\) by profile.

\end{definition}

\begin{proposition}[Taylor duality]

\label{prop:taylor-duality} Let \(X, Y\) be abelian groups,
\(g: X \to Y\), \(x \in X\), \(u_\bullet = (u_1, \dots, u_k) \in X^k\),
and \(y = g(x)\).

\begin{enumerate}
\def\labelenumi{\arabic{enumi})}
\tightlist
\item
  \emph{Boolean Taylor duality.}
\end{enumerate}

\[
  g(x + {\textstyle\sum_{i=1}^k} u_i) = y + \sum_{\emptyset \neq S \subseteq [k]} \Delta(g; x; u_S), \qquad \Delta(g; x; u_\bullet) = \sum_{S \subseteq [k]} (-1)^{k-|S|}\, g(x + {\textstyle\sum_{i \in S}} u_i).
\]

\begin{enumerate}
\def\labelenumi{\arabic{enumi})}
\setcounter{enumi}{1}
\tightlist
\item
  \emph{Binomial Taylor duality.} For \(\alpha \in \IN_0^k\):
\end{enumerate}

\[
  g(x + {\textstyle\sum_{i=1}^k} \alpha_i u_i) = y + \sum_{0 < \beta \leq \alpha} \frac{(\alpha)_\beta}{\beta!}\, \Delta(g; x; u_\bullet^\beta), \qquad \Delta(g; x; u_\bullet^\alpha) = \sum_{\beta \leq \alpha} (-1)^{|\alpha|-|\beta|}\, \frac{(\alpha)_\beta}{\beta!}\, g(x + {\textstyle\sum_{i=1}^k} \beta_i u_i).
\]

\end{proposition}

\begin{proof}

\begin{enumerate}
\def\labelenumi{\arabic{enumi})}
\item
  The second identity is the definition of \(\Delta\). The first is its
  Boolean Möbius inverse \ref{prop:moebius-inversion}.
\item
  The second identity is the grouped form of \(\Delta^\alpha\)
  \ref{def:forward-diff}. The first is its Binomial Möbius inverse
  \ref{prop:multi-index-moebius}.
\end{enumerate}

\end{proof}

\subsection{Faà di Bruno duality}\label{fauxe0-di-bruno-duality}

Write \(\KP_+(k)\) for the set of nonempty subsets of \([k]\),
\(\KP_+^2(k) = \KP_+(\KP_+(k))\) for the set of nonempty families of
such subsets, and
\(\Cov(k) = \set{H \in \KP_+^2(k) \mid \bigcup H = [k]}\) for the set of
all coverings of \([k]\) by nonempty subsets.

\begin{theorem}[Discrete Faà di Bruno Duality]

\label{thm:dfdb}

Let \(X, Y, Z\) be abelian groups, \(g: X \to Y\), \(f: Y \to Z\)
arbitrary maps, \(x \in X\), \(u_1, \dots, u_k \in X\). Write
\(y = g(x)\) and \(z = f(y)\). Then:

\begin{enumerate}
\def\labelenumi{\arabic{enumi})}
\tightlist
\item
  \emph{Boolean Faà di Bruno} (\(k \geq 1\))\emph{:}
\end{enumerate}

\[
  \Delta(f \circ g; x; u_\bullet)
  =
  \sum_{H \in \Cov(k)}
  \Delta(f; y; (\Delta(g; x; u_T))_{T \in H}).
\]

\begin{enumerate}
\def\labelenumi{\arabic{enumi})}
\setcounter{enumi}{1}
\tightlist
\item
  \emph{Boolean Taylor composition:}
\end{enumerate}

\[
  (f \circ g)(x + {\textstyle\sum_{i=1}^k} u_i) = z + \sum_{H \in \KP_+^2(k)} \Delta(f; y; (\Delta(g; x; u_T))_{T \in H}).
\]

\end{theorem}

\begin{proof}

\begin{enumerate}
\def\labelenumi{\arabic{enumi})}
\setcounter{enumi}{1}
\item
  Write \(g_T := \Delta(g; x; u_T)\). By Taylor duality
  \ref{prop:taylor-duality},
  \(T(g; x; u_S) = y + \sum_{\emptyset \neq T \subseteq S} g_T\).
  Applying Taylor duality to \(f\) at \(y\) in directions
  \((g_T)_{T \in \KP_+(S)}\) gives
  \(T(f \circ g; x; u_S) = z + \sum_{H \in \KP_+^2(S)} \Delta(f; y; (g_T)_{T \in H})\).
\item
  Define
  \(\varphi(S) := \sum_{H \in \Cov(S)} \Delta(f; y; (g_T)_{T \in H})\).
  By Möbius inversion \ref{prop:moebius-inversion}, it suffices to show
  \(\zeta(\varphi; S) = T(f \circ g; x; u_S) - z\). Indeed,
  \(\zeta(\varphi; S) = \sum_{R \subseteq S} \sum_{H \in \Cov(R)} \Delta(f; y; (g_T)_{T \in H}) = \sum_{H \in \KP_+^2(S)} \Delta(f; y; (g_T)_{T \in H})\),
  since each \(H \in \KP_+^2(S)\) appears in exactly one summand, namely
  the one for \(R = \bigcup H\). Now apply part 2.
\end{enumerate}

\end{proof}

\begin{remark}[Binomial Faà di Bruno]

The Boolean Faà di Bruno formula extends naturally to a binomial version
in terms of multi-indices \(\gamma \in \IN_0^S\), by applying the
Boolean formula to the slot set \(S(\gamma)\) and grouping by profile.
The resulting expression resembles the Constantine--Savits form
\citep{CS1996}; the precise coefficients and a recursive version are
given in \ref{thm:iterated-dfdb} below.

\end{remark}

\subsection{Iterated Faà di Bruno}\label{iterated-fauxe0-di-bruno}

The Boolean Faà di Bruno formula extends to \(m\)-fold compositions,
indexed by \(\Cov_m(k)\). We first set up the combinatorial apparatus:
higher power sets, higher multi-indices, and higher partitions.

\begin{definition}[Higher power sets and coverings]

\label{def:coverings} Recall that \(\KP(S)\) is the power set and
\(\KP_+(S) = \KP(S) \setminus \set{\emptyset}\) the set of nonempty
subsets.

\begin{itemize}
\tightlist
\item
  The \emph{higher power sets} are \(\KP_+^0(S) = S\),
  \(\KP_+^1(S) = \KP_+(S)\), and \(\KP_+^r(S) := \KP_+(\KP_+^{r-1}(S))\)
  for \(r \geq 2\).
\item
  For \(K \in \KP_+^r(S)\), the \emph{leaf support}
  \(\lf(K) \subseteq S\) is defined recursively: \(\lf(K) := K\) for
  \(r = 1\) and \(\lf(K) := \bigcup_{L \in K} \lf(L)\) for \(r \geq 2\).
\item
  We say \(K\) \emph{covers} \(S\) if \(\lf(K) = S\), and write
  \(\Cov_r(S) := \set{K \in \KP_+^r(S) \mid \lf(K) = S}\) for the set of
  \emph{\(r\)-fold coverings}.
\end{itemize}

For \(S = [k]\), we write \(\KP_+^r(k)\) and \(\Cov_r(k)\).

\end{definition}

\begin{definition}[Higher multi-indices]

\label{def:higher-multi-indices} For a set \(S\), let \(\KM(S)\) be the
set of finitely supported maps \(S \to \IN_0\), and
\(\KM_+(S) = \KM(S) \setminus \set{0}\). For finite \(S\),
\(\KM(S) = \IN_0^S\).

\begin{itemize}
\tightlist
\item
  The \emph{iterated multiset sets} are \(\KM_+^0(S) = S\),
  \(\KM_+^1(S) = \KM_+(S)\), and \(\KM_+^{r+1}(S) := \KM_+(\KM_+^r(S))\)
  for \(r \geq 1\).
\item
  For \(\kappa \in \KM_+^r(S)\) with \(r \geq 1\), the \emph{support} is
  \(\supp(\kappa) := \set{\lambda \in \KM_+^{r-1}(S) : \kappa(\lambda) > 0}\).
\item
  The \emph{leaf multi-index} \(\lf(\kappa) \in \IN_0^S\) is defined
  recursively: \(\lf(\alpha) := \alpha\) for \(r = 1\) and
  \(\lf(\kappa) := \sum_{\lambda \in \supp(\kappa)} \kappa(\lambda)\, \lf(\lambda)\)
  for \(r \geq 2\).
\end{itemize}

The \emph{higher profile map} \(\nu: \KP_+^m(S(\gamma)) \to \KM_+^m(S)\)
is defined recursively: \(\nu(T) \in \IN_0^S\) for \(m = 1\) as in
\ref{def:multi-index}, and
\(\nu(K)(\lambda) := \#\set{L \in K : \nu(L) = \lambda}\) for
\(m \geq 2\).

The embedding \(\KP_+(S) \hookrightarrow \KM_+(S)\) via
\(T \mapsto 1_T\) extends to \(\KP_+^r(S) \hookrightarrow \KM_+^r(S)\)
at every level: the Boolean case is the height-one restriction of the
multi-index iteration.

\end{definition}

\begin{definition}[Higher partitions]

\label{def:higher-partitions}

\begin{itemize}
\tightlist
\item
  A \emph{partition} of a finite set \(S\) is a set
  \(\pi = \set{B_1, \dots, B_r}\) of nonempty pairwise disjoint subsets
  with \(B_1 \sqcup \cdots \sqcup B_r = S\). Write \(\Part(S)\) for the
  set of partitions and \(\Part(k)\) for \(\Part([k])\).
\item
  \emph{Higher partitions} are defined recursively:
  \(\Part_1(S) = \set{S}\), and for \(m \geq 1\),
  \(\Part_{m+1}(S) := \set{\set{H_B}_{B \in \pi} \mid \pi \in \Part(S),\; H_B \in \Part_m(B)}\).
  For \(m = 2\), \(\Part_2(S) = \Part(S)\).
\item
  The \emph{weight} of \(H \in \KP_+^r(S)\) is \(\mathrm{wt}(H) := |H|\)
  for \(r = 1\) and \(\mathrm{wt}(H) := \sum_{K \in H} \mathrm{wt}(K)\)
  for \(r \geq 2\).
\item
  A \emph{multi-index partition} of \(\gamma \in \IN_0^S\) is an
  iterated multi-index \(\kappa \in \KM_+^m(S)\) with
  \(\lf(\kappa) = \gamma\); we write \(\kappa \vdash \gamma\). For
  \(m = 2\) this is a multiset \(\kappa\) of nonzero multi-indices with
  \(\sum_\beta \kappa(\beta)\, \beta = \gamma\).
\end{itemize}

\end{definition}

\begin{lemma}[Weight bound]

\label{lem:weight-bound} For any covering \(H \in \Cov_r(S)\) with
\(r \geq 1\), \(\mathrm{wt}(H) \geq |S|\), with equality if and only if
\(H \in \Part_r(S)\).

\end{lemma}

\begin{proof}

For \(r = 1\): \(\Cov_1(S) = \Part_1(S) = \set{S}\) and
\(\mathrm{wt}(S) = |S|\). For \(r = 2\):
\(\mathrm{wt}(H) = \sum_{T \in H} |T| \geq |\bigcup H| = |S|\), with
equality iff the blocks are pairwise disjoint. For \(r \geq 3\): by
induction,
\(\mathrm{wt}(H) = \sum_{K \in H} \mathrm{wt}(K) \geq \sum_{K \in H} |\lf(K)| \geq |S|\),
with equality iff each \(K \in \Part_{r-1}(\lf(K))\) and the leaf
supports partition \(S\).

\end{proof}

\begin{definition}[Iterated increments]

\label{def:iterated-incr} Let
\(X_0 \xrightarrow{f_1} X_1 \xrightarrow{f_2} \cdots \xrightarrow{f_m} X_m\)
be maps between abelian groups, \(x \in X_0\), \(u: S \to X_0\) a
direction family, and \(x_r = (f_r \circ \cdots \circ f_1)(x)\).

\begin{enumerate}
\def\labelenumi{\arabic{enumi})}
\tightlist
\item
  \emph{Boolean iterated increments.} For \(r = 1\) and
  \(T \in \KP_+(S)\), set
  \(\Delta^{T}(f_1;\, x;\, u) := \Delta(f_1;\, x;\, u_T)\). For
  \(r \geq 2\) and \(K \in \KP_+^{r}(S)\), define recursively:
\end{enumerate}

\[
  \Delta^{K}(f_1, \dots, f_r;\, x;\, u)
  :=
  \Delta(f_r;\, x_{r-1};\,
  (\Delta^{L}(f_1, \dots, f_{r-1};\, x;\, u))_{L \in K}).
\]

\begin{enumerate}
\def\labelenumi{\arabic{enumi})}
\setcounter{enumi}{1}
\tightlist
\item
  \emph{Binomial iterated increments.} For \(r = 1\) and
  \(\alpha \in \KM_+(S)\), set
  \(\Delta^{\alpha}(f_1;\, x;\, u) := \Delta(f_1;\, x;\, u^{\alpha})\)
  \ref{def:forward-diff}. For \(r \geq 2\) and
  \(\kappa \in \KM_+^{r}(S)\), define recursively:
\end{enumerate}

\[
  \Delta^{\kappa}(f_1, \dots, f_r;\, x;\, u)
  :=
  \Delta(f_r;\, x_{r-1};\,
  (\Delta^{\lambda}(f_1, \dots, f_{r-1};\, x;\, u)^{\times \kappa(\lambda)})_{\lambda \in \supp(\kappa)}),
\]

where for each \(\lambda\) with \(\kappa(\lambda) > 0\), the direction
\(\Delta^{\lambda}(\dots)\) appears \(\kappa(\lambda)\) times. The
Boolean case is recovered by restriction to
\(\KP_+^r(S) \subset \KM_+^r(S)\).

\end{definition}

\begin{lemma}[Profile invariance]

\label{lem:profile-invariance} Let \(\gamma \in \IN_0^S\), let
\(u: S \to X_0\) be a direction family, and equip \(S(\gamma)\) with the
slot directions \(u \circ \pi\). Then for every
\(K \in \KP_+^m(S(\gamma))\),

\[
  \Delta^K(f_1, \dots, f_m;\, x;\, u \circ \pi) = \Delta^{\nu(K)}(f_1, \dots, f_m;\, x;\, u).
\]

\end{lemma}

\begin{proof}

For \(m = 1\) and \(T \subseteq S(\gamma)\), the direction list
\((u_{\pi(t)})_{t \in T}\) is a permutation of \(u^{\nu(T)}\), so the
claim is permutation invariance of \(\Delta\) \ref{def:forward-diff}.
For \(m \geq 2\), apply the induction hypothesis to each \(L \in K\):
\(\Delta^L = \Delta^{\nu(L)}\). In the outer difference,
\(\kappa(\lambda)\) of the directions equal \(\Delta^\lambda\), which is
the defining expression for \(\Delta^{\nu(K)}\).

\end{proof}

\begin{theorem}[Iterated Faà di Bruno]

\label{thm:iterated-dfdb} Let
\(X_0 \xrightarrow{f_1} X_1 \xrightarrow{f_2} \cdots \xrightarrow{f_m} X_m\)
be arbitrary maps between abelian groups, \(x \in X_0\), and
\(z = (f_m \circ \cdots \circ f_1)(x)\). Let \(S\) be a finite set,
\(u: S \to X_0\) a family of directions, and \(\gamma \in \IN_0^S\);
write \(u_S = (u_s)_{s \in S}\), and \(u^\gamma\) for the family with
each \(u_s\) repeated \(\gamma_s\) times. For \(S = [k]\) this recovers
the tuple notation \(u_\bullet = (u_1, \dots, u_k)\).

\begin{enumerate}
\def\labelenumi{\arabic{enumi})}
\tightlist
\item
  \emph{Boolean Faà di Bruno} (\(S \neq \emptyset\))\emph{:}
\end{enumerate}

\[
  \Delta(f_m \circ \cdots \circ f_1;\, x;\, u_S)
  =
  \sum_{K \in \Cov_m(S)}
  \Delta^{K}(f_1, \dots, f_m;\, x;\, u).
\]

\begin{enumerate}
\def\labelenumi{\arabic{enumi})}
\setcounter{enumi}{1}
\tightlist
\item
  \emph{Binomial Faà di Bruno:}
\end{enumerate}

\[
  \Delta(f_m \circ \cdots \circ f_1; x; u^\gamma) = \sum_{\kappa \in \KM_+^m(S)} \Cov_m(\gamma, \kappa)\, \Delta^{\kappa}(f_1, \dots, f_m; x; u),
\]

where
\(\Cov_m(\gamma, \kappa) = \#\set{K \in \Cov_m(S(\gamma)) : \nu(K) = \kappa}\)
counts \(m\)-fold coverings of \(S(\gamma)\) with profile \(\kappa\).

\begin{enumerate}
\def\labelenumi{\arabic{enumi})}
\setcounter{enumi}{2}
\tightlist
\item
  \emph{Boolean Taylor composition:}
\end{enumerate}

\[
  (f_m \circ \cdots \circ f_1)(x + {\textstyle\sum_{s \in S}} u_s) = z + \sum_{K \in \KP_+^m(S)} \Delta^{K}(f_1, \dots, f_m;\, x;\, u).
\]

\begin{enumerate}
\def\labelenumi{\arabic{enumi})}
\setcounter{enumi}{3}
\tightlist
\item
  \emph{Binomial Taylor composition:}
\end{enumerate}

\[
  (f_m \circ \cdots \circ f_1)(x + {\textstyle\sum_s} \gamma_s u_s) = z + \sum_{\kappa \in \KM_+^m(S)} \mathrm{Pow}_m(\gamma, \kappa)\, \Delta^{\kappa}(f_1, \dots, f_m; x; u),
\]

where
\(\mathrm{Pow}_m(\gamma, \kappa) = \#\set{K \in \KP_+^m(S(\gamma)) : \nu(K) = \kappa}\)
counts \(m\)-fold iterated subsets of \(S(\gamma)\) with profile
\(\kappa\).

\end{theorem}

\begin{proof}

\begin{enumerate}
\def\labelenumi{\arabic{enumi})}
\setcounter{enumi}{2}
\tightlist
\item
  By induction on \(m\). The cases \(m = 1\) and \(m = 2\) are
  \ref{prop:taylor-duality} and \ref{thm:dfdb}. Write
  \(F = f_m \circ \cdots \circ f_1\),
  \(F' = f_{m-1} \circ \cdots \circ f_1\),
  \(x_r = (f_r \circ \cdots \circ f_1)(x)\), and
  \(d_K := \Delta^K(f_1, \dots, f_{m-1};\, x;\, u)\) for
  \(K \in \KP_+^{m-1}(S)\). The induction hypothesis gives
  \(F'(x + \sum_{s \in S} u_s) = x_{m-1} + \sum_{K \in \KP_+^{m-1}(S)} d_K\).
  Applying Taylor duality \ref{prop:taylor-duality} to \(f_m\) at
  \(x_{m-1}\) in directions \(d_K\):
\end{enumerate}

\[
  F(x + {\textstyle\sum_{s \in S}} u_s) = x_m + \sum_{H \in \KP_+(\KP_+^{m-1}(S))} \Delta(f_m;\, x_{m-1};\, (d_K)_{K \in H}).
\]

Since \(\KP_+(\KP_+^{m-1}(S)) = \KP_+^m(S)\) and
\(\Delta(f_m;\, x_{m-1};\, (d_K)_{K \in H}) = \Delta^H(f_1, \dots, f_m;\, x;\, u)\)
by \ref{def:iterated-incr}, this completes the induction.

\begin{enumerate}
\def\labelenumi{\arabic{enumi})}
\tightlist
\item
  Define \(\varphi(R) := \sum_{H \in \Cov_m(R)} \Delta^H\) for
  \(R \subseteq S\). By Möbius inversion, it suffices to show
  \(\zeta(\varphi; R) = T(F; x; u_R) - z\). Indeed,
  \(\zeta(\varphi; R) = \sum_{Q \subseteq R} \sum_{H \in \Cov_m(Q)} \Delta^H = \sum_{H \in \KP_+^m(R)} \Delta^H\),
  since each \(H \in \KP_+^m(R)\) covers exactly \(Q = \lf(H)\). Now
  apply part 3.
\end{enumerate}

2,4) Apply parts 1,3 to \(S(\gamma)\). By profile invariance
\ref{lem:profile-invariance}, \(\Delta^{K}\) depends only on the profile
\(\kappa = \nu(K)\), so grouping gives
\(\sum_\kappa \Cov_m(\gamma, \kappa)\, \Delta^{\kappa}\) for the FdB
side and
\(\sum_\kappa \mathrm{Pow}_m(\gamma, \kappa)\, \Delta^{\kappa}\) for the
Taylor side.

\end{proof}

\begin{proposition}[Newton coefficient recursion]

\label{prop:newton-recursion} For \(\alpha \in \KM_+(S)\) and
\(\kappa \in \KM_+^m(S)\) with \(m \geq 2\):

\[
  \mathrm{Pow}_1(\gamma, \alpha) = \frac{(\gamma)_\alpha}{\alpha!}, \qquad \mathrm{Pow}_2(\gamma, \kappa) = \frac{1}{\kappa!} \prod_{\alpha \in \supp(\kappa)} \left(\frac{(\gamma)_\alpha}{\alpha!}\right)_{\kappa(\alpha)},
\]

\[
  \mathrm{Pow}_m(\gamma, \kappa) = \frac{1}{\kappa!} \prod_{\lambda \in \supp(\kappa)} (\mathrm{Pow}_{m-1}(\gamma, \lambda))_{\kappa(\lambda)}.
\]

\end{proposition}

\begin{proof}

For \(m = 1\):
\(\mathrm{Pow}_1(\gamma, \alpha) = \#\set{T \subseteq S(\gamma) : \nu(T) = \alpha} = \frac{(\gamma)_\alpha}{\alpha!}\),
counting subsets with profile \(\alpha\).

For \(m \geq 2\): an element \(K \in \KP_+^m(S(\gamma))\) is a nonempty
set of elements of \(\KP_+^{m-1}(S(\gamma))\). The profile
\(\kappa \in \KM_+^m(S)\) records
\(\kappa(\lambda) = \#\set{\text{elements of } K \text{ with profile } \lambda}\).
Since the elements are chosen as a set (no repeats) from
\(\mathrm{Pow}_{m-1}(\gamma, \lambda)\) available elements with profile
\(\lambda\), the count is
\(\prod_\lambda (\mathrm{Pow}_{m-1}(\gamma, \lambda))_{\kappa(\lambda)} / \kappa(\lambda)!\).

\end{proof}

\begin{proposition}[Covering coefficient recursion]

\label{prop:covering-recursion}

\[
  \Cov_1(\gamma, \alpha) = \delta_{\gamma, \alpha}, \qquad \Cov_2(\gamma, \kappa) = \sum_{\beta \leq \gamma} (-1)^{\mathrm{wt}(\gamma - \beta)} \frac{(\gamma)_\beta}{\beta!} \frac{1}{\kappa!} \prod_{\alpha \in \supp(\kappa)} \left(\frac{(\beta)_\alpha}{\alpha!}\right)_{\kappa(\alpha)},
\]

and generally, for \(m \geq 2\):

\begin{enumerate}
\def\labelenumi{\arabic{enumi})}
\tightlist
\item
  \emph{Cross recursion} (\(\Cov_m\) from \(\mathrm{Pow}_m\)):
\end{enumerate}

\[
  \Cov_m(\gamma, \kappa) = \sum_{\beta \leq \gamma} (-1)^{\mathrm{wt}(\gamma - \beta)} \frac{(\gamma)_\beta}{\beta!}\, \mathrm{Pow}_m(\beta, \kappa).
\]

\begin{enumerate}
\def\labelenumi{\arabic{enumi})}
\setcounter{enumi}{1}
\tightlist
\item
  \emph{Level recursion} (\(\Cov_m\) from \(\Cov_{m-1}\)):
\end{enumerate}

\[
  \Cov_m(\gamma, \kappa) = \sum_{\beta \leq \gamma} (-1)^{\mathrm{wt}(\gamma - \beta)} \frac{(\gamma)_\beta}{\beta!} \frac{1}{\kappa!} \prod_{\lambda \in \supp(\kappa)} \left(\sum_{\alpha \leq \beta} \frac{(\beta)_\alpha}{\alpha!}\, \Cov_{m-1}(\alpha, \lambda)\right)_{\!\kappa(\lambda)}.
\]

\end{proposition}

\begin{proof}

\begin{enumerate}
\def\labelenumi{\arabic{enumi})}
\item
  An element \(K \in \KP_+^m(S(\gamma))\) with profile \(\kappa\) has
  leaf support \(\lf(K) \subseteq S(\gamma)\). Each subset
  \(R \subseteq S(\gamma)\) with \(\nu(R) = \beta\) supports
  \(\Cov_m(\beta, \kappa)\) coverings with profile \(\kappa\), and there
  are \(\frac{(\gamma)_\beta}{\beta!}\) such subsets. Hence
  \(\mathrm{Pow}_m(\gamma, \kappa) = \sum_{\beta \leq \gamma} \frac{(\gamma)_\beta}{\beta!}\, \Cov_m(\beta, \kappa)\).
  Binomial Möbius inversion \ref{prop:multi-index-moebius} gives
  \(\Cov_m(\gamma, \kappa) = \sum_{\beta \leq \gamma} (-1)^{\mathrm{wt}(\gamma-\beta)} \frac{(\gamma)_\beta}{\beta!}\, \mathrm{Pow}_m(\beta, \kappa)\).
\item
  Substitute the Newton recursion \ref{prop:newton-recursion}
  \(\mathrm{Pow}_m(\beta, \kappa) = \frac{1}{\kappa!} \prod_\lambda (\mathrm{Pow}_{m-1}(\beta, \lambda))_{\kappa(\lambda)}\)
  into 1), then replace
  \(\mathrm{Pow}_{m-1}(\beta, \lambda) = \sum_{\alpha \leq \beta} \frac{(\beta)_\alpha}{\alpha!}\, \Cov_{m-1}(\alpha, \lambda)\),
  which is the zeta inverse of 1) at level \(m-1\).
\end{enumerate}

\end{proof}

\section{Polynomial Möbius
calculus}\label{polynomial-muxf6bius-calculus}

This section passes from the discrete setting to the differential
setting on polynomial rings. The discrete formulas of the previous
section are algebra over the Boolean cube algebra \(B_k\); the
differential formulas are algebra over the infinitesimal Taylor algebra
\(A_k\). A flat deformation family \(C_k\) interpolates between them.
The key objects (Taylor--Möbius duality, Faà di Bruno by composition,
and the Constantine--Savits coefficients) all transfer from \(B_k\) to
\(A_k\) with \(\Delta\) replaced by \(D\) and coverings replaced by
partitions.

Throughout this section, let \(\ik\) be a commutative ring. The
square-free results (\ref{lem:flatness} at \(\nu = \mathbf{1}\),
\ref{prop:deformation-taylor}, \ref{thm:cubical-fdb},
\ref{thm:fdb-deformation}) hold over any \(\ik\); the binomial
refinements require \(\IQ \subseteq \ik\) for the factorial
denominators.

\subsection{Taylor algebras}\label{taylor-algebras}

\begin{definition}[Taylor algebras]

\label{def:taylor-algebras} For \(\nu \in \IN_0^k\), define:

\[
  B_k^\nu := \ik[\delta_1, \dots, \delta_k] / (\delta_i)_{\nu_i+1}, \qquad A_k^\nu := \ik[\eps_1, \dots, \eps_k] / (\eps_i^{\nu_i+1}).
\]

Here
\((\delta_i)_{\nu_i+1} = \delta_i(\delta_i - 1) \cdots (\delta_i - \nu_i)\)
is the falling factorial, and relations carrying a free index \(i\)
generate one relation for each \(1 \leq i \leq k\). We call \(B_k^\nu\)
the \emph{grid algebra} (for \(\IQ \subseteq \ik\), the function algebra
on the finite grid \(\prod_i \set{0, \dots, \nu_i}\); for
\(\nu = \mathbf{1}\) this holds over any \(\ik\)) and \(A_k^\nu\) the
\emph{Taylor algebra} (its infinitesimal version).

\end{definition}

\begin{example}[$\nu = \mathbf{1}$]

\label{ex:cube-algebras} For \(\nu = (1, \dots, 1)\), the relations
simplify to \(\delta_i^2 = \delta_i\) and \(\eps_i^2 = 0\). Write
\(B_k = B_k^{\mathbf{1}}\) and \(A_k = A_k^{\mathbf{1}}\). Both are free
of rank \(2^k\) with basis \(\delta^S := \prod_{i \in S} \delta_i\)
resp. \(\eps^S := \prod_{i \in S} \eps_i\) for \(S \subseteq [k]\). The
multiplication rules are:

\begin{itemize}
\tightlist
\item
  \(\delta^S \delta^T = \delta^{S \cup T}\) in \(B_k\) (idempotent:
  overlaps absorbed).
\item
  \(\eps^S \eps^T = \eps^{S \cup T}\) if \(S \cap T = \emptyset\), and
  \(\eps^S \eps^T = 0\) otherwise, in \(A_k\) (nilpotent: overlapping
  products vanish).
\end{itemize}

A product
\(\delta^{T_1} \cdots \delta^{T_p} = \delta^{T_1 \cup \cdots \cup T_p}\)
contributes to the top face \(\delta^{[k]}\) exactly when
\(T_1 \cup \cdots \cup T_p = [k]\), i.e.~a covering. A product
\(\eps^{T_1} \cdots \eps^{T_p}\) is nonzero only when the \(T_i\) are
pairwise disjoint, i.e.~a partition.

\end{example}

\subsection{Taylor--Möbius duality}\label{taylormuxf6bius-duality}

\begin{proposition}[Boolean Taylor--Möbius duality]

\label{prop:poly-taylor-duality} Let \(p \in \ik[y_1, \dots, y_d]\),
\(x \in \ik^d\), and \(v_1, \dots, v_k \in \ik^d\). In \(B_k\):

\[
  p(x + {\textstyle\sum_i} \delta_i v_i) = \sum_{S \subseteq [k]} \Delta(p; x; v_S)\, \delta^S, \qquad \Delta(p; x; v_S)\, \delta^S = \sum_{T \subseteq S} (-1)^{|S|-|T|}\, p(x + {\textstyle\sum_{i \in T}} \delta_i v_i).
\]

In \(A_k\):

\[
  p(x + {\textstyle\sum_i} \eps_i v_i) = \sum_{S \subseteq [k]} D(p; x; v_S)\, \eps^S, \qquad \eps^S\, D(p; x; v_S) = \sum_{T \subseteq S} (-1)^{|S|-|T|}\, p(x + {\textstyle\sum_{i \in T}} \eps_i v_i).
\]

Here \(D(p; x; v_S)\) denotes the \(\eps^S\)-coefficient of
\(p(x + \sum_i v_i \eps_i)\) in \(A_k\); this agrees with the iterated
directional derivative when \(\ik \supseteq \IQ\).

\end{proposition}

\begin{proof}

In \(B_k\): expand
\(p(x + \sum_i \delta_i v_i) = \sum_S c_S\, \delta^S\) in the free
basis. The branch maps \(\rho_T: \delta_i \mapsto 1_{i \in T}\) give
\(p(x + \sum_{i \in T} v_i) = \sum_{S \subseteq T} c_S\); Boolean Möbius
inversion \ref{prop:moebius-inversion} gives
\(c_S = \Delta(p; x; v_S)\). The Möbius identity is the same inversion
applied inside \(B_k\): the right side is
\(\sum_{T \subseteq S} (-1)^{|S|-|T|}\, p(x + \sum_{i \in T} \delta_i v_i) = \sum_{T \subseteq S} (-1)^{|S|-|T|} \sum_{R \subseteq T} c_R\, \delta^R = c_S\, \delta^S\).
In \(A_k\): the \(\eps^S\)-coefficient is \(D(p; x; v_S)\) by the
multinomial theorem; the Möbius identity says the same alternating sieve
recovers differentials from infinitesimal grid evaluations.

\end{proof}

\begin{proposition}[Binomial Taylor--Möbius duality]

\label{prop:binomial-taylor-duality} Let \(p \in \ik[y_1, \dots, y_d]\),
\(x \in \ik^d\), and \(v_1, \dots, v_k \in \ik^d\). In \(B_k^\nu\):

\[
  p(x + {\textstyle\sum_i} \delta_i v_i) = \sum_{0 \leq \alpha \leq \nu} \frac{1}{\alpha!}\, \Delta^\alpha(p; x; v_\bullet)\, (\delta)_\alpha, \qquad \Delta^\alpha(p; x; v_\bullet) = \sum_{\beta \leq \alpha} (-1)^{\mathrm{wt}(\alpha-\beta)}\, \frac{(\alpha)_\beta}{\beta!}\, p(x + {\textstyle\sum_i} \beta_i v_i).
\]

In \(A_k^\nu\):

\[
  p(x + {\textstyle\sum_i} \eps_i v_i) = \sum_{0 \leq \alpha \leq \nu} \frac{1}{\alpha!}\, D^\alpha(p; x; v_\bullet)\, \eps^\alpha, \qquad \eps^\nu\, D^\nu(p; x; v_\bullet) = \sum_{\beta \leq \nu} (-1)^{\mathrm{wt}(\nu-\beta)}\, \frac{(\nu)_\beta}{\beta!}\, p(x + {\textstyle\sum_i} \beta_i \eps_i v_i).
\]

Here
\(D^\alpha(p; x; v_\bullet) := D(p; x; v_\bullet^{\times \alpha})\), the
square-free \(D\) applied to the slot realization
\ref{def:forward-diff}. By the multinomial theorem,
\(\alpha!\, [\eps^\alpha]\, p(x + \sum_i \eps_i v_i) = D(p; x; v_\bullet^{\times \alpha})\);
when \(\IQ \subseteq \ik\), both sides equal the iterated directional
derivative \((D_{v_1}^{\alpha_1} \cdots D_{v_k}^{\alpha_k} p)(x)\).

\end{proposition}

\begin{proof}

In \(B_k^\nu\): the branch maps \(\rho_\beta: B_k^\nu \to \ik\) send
\(\delta_i \mapsto \beta_i\) for \(0 \leq \beta \leq \nu\). Since
\(\rho_\beta((\delta)_\alpha) = (\beta)_\alpha\), the expansion
coefficients are determined by binomial Möbius inversion
\ref{prop:multi-index-moebius}, giving \(\Delta^\alpha / \alpha!\). The
Möbius identity is the definition of \(\Delta^\alpha\).

In \(A_k^\nu\): the expansion gives \(D^\alpha / \alpha!\) as the
\(\eps^\alpha\)-coefficient by the multinomial theorem. For the Möbius
identity: the coefficient of \(D^\alpha \eps^\alpha / \alpha!\) in the
alternating sum is \(\prod_i \Delta^{\nu_i}(x^{\alpha_i})(0)\), which
annihilates \(\alpha_i < \nu_i\) and gives \(\nu_i!\) at
\(\alpha_i = \nu_i\).

\end{proof}

\subsection{Faà di Bruno duality}\label{fauxe0-di-bruno-duality-1}

\begin{theorem}[Boolean Faà di Bruno]

\label{thm:cubical-fdb} Let \(q: \ik^e \to \ik^d\) and
\(p: \ik^d \to \ik\) be polynomials, \(x \in \ik^e\),
\(v_1, \dots, v_k \in \ik^e\), and \(y = q(x)\).

\[
  \Delta(p \circ q; x; v_{[k]}) = \sum_{\substack{H \subseteq \KP_+([k]) \\\\ \bigcup H = [k]}} \Delta(p; y; (\Delta(q; x; v_T))_{T \in H}),
\]

\[
  D(p \circ q; x; v_{[k]}) = \sum_{\pi \in \Part([k])} D(p; y; (D(q; x; v_B))_{B \in \pi}).
\]

\end{theorem}

\begin{proof}

Write \(q_T := \Delta(q; x; v_T)\) for
\(\emptyset \neq T \subseteq [k]\). By cubical Taylor
\ref{prop:poly-taylor-duality},
\(q(x + \sum_i \delta_i v_i) = y + \sum_{\emptyset \neq T} q_T\, \delta^T\)
in \(B_k\). Since \((\delta^T)^2 = \delta^T\), the map
\(\delta_T \mapsto \delta^T\) defines an algebra map \(B_l \to B_k\)
where \(l = 2^k - 1\). Applying cubical Taylor to \(p\) at \(y\) in the
\(l\) directions \(q_T\) via this map:

\[
  (p \circ q)(x + {\textstyle\sum_i} \delta_i v_i) = p(y + {\textstyle\sum_T} q_T \delta^T) = \sum_{S \subseteq [k]} \sum_{\substack{H \subseteq \KP_+(S) \\\\ \bigcup H = S}} \Delta(p; y; (q_T)_{T \in H})\, \delta^S.
\]

Extracting \([\delta^{[k]}]\) gives the \(B_k\) formula: the covering
condition \(\bigcup H = [k]\) comes from
\(\delta^{T_1} \cdots \delta^{T_r} = \delta^{T_1 \cup \cdots \cup T_r}\).

The same argument in \(A_k\): since \((\eps^T)^2 = 0\), the map
\(\eps_T \mapsto \eps^T\) defines \(A_l \to A_k\). Now
\(\eps^{T_1} \cdots \eps^{T_r} = 0\) unless the \(T_i\) are pairwise
disjoint, so coverings reduce to partitions.

\end{proof}

\begin{remark}

The \(B_k\) formula is the algebraic form of the discrete Faà di Bruno
\ref{thm:dfdb}. Since
\(B_k \otimes Y \cong \mathrm{Map}(\set{0,1}^k, Y)\) and arbitrary maps
act vertexwise, the \(B_k\) proof above is the same proof as
\ref{thm:dfdb} expressed in the coordinate ring of the cube.

\end{remark}

\subsection{Iterated Faà di Bruno}\label{iterated-fauxe0-di-bruno-1}

\begin{proposition}[Binomial Faà di Bruno / Constantine--Savits]

\label{prop:poly-fdb} Let \(q: \ik^e \to \ik^d\) and
\(p: \ik^d \to \ik\) be polynomials, \(x \in \ik^e\),
\(v_1, \dots, v_k \in \ik^e\), \(y = q(x)\), and \(\gamma \in \IN_0^k\):

\[
  D(p \circ q;\, x;\, v_\bullet^{\times \gamma}) = \sum_{\kappa \vdash \gamma} \frac{\gamma!}{\kappa!\, \prod_\alpha (\alpha!)^{\kappa(\alpha)}}\, D(p;\, y;\, (D(q;\, x;\, v_\bullet^{\times \alpha}))_\alpha^{\times \kappa(\alpha)}).
\]

\end{proposition}

\begin{proof}

Apply the \(A_k\) Faà di Bruno \ref{thm:cubical-fdb} to \(S(\gamma)\)
with \(|\gamma|\) directions where \(v_i\) is repeated \(\gamma_i\)
times. Since interchanging repeated copies of \(v_i\) does not change
\(D^H\), we group partitions by their multi-index profile \(\kappa\).
The number of partitions with profile \(\kappa\) is
\(\Part_2(\gamma, \kappa) = \gamma! / (\kappa!\, \prod_\alpha (\alpha!)^{\kappa(\alpha)})\).

\end{proof}

\begin{definition}[Iterated differentials]

\label{def:poly-iterated-diff} For polynomial maps \(f_1, \dots, f_m\)
with \(f_r: \ik^{d_{r-1}} \to \ik^{d_r}\), \(x \in \ik^{d_0}\),
\(v_1, \dots, v_k \in \ik^{d_0}\), and
\(x_r = (f_r \circ \cdots \circ f_1)(x)\), define the \emph{iterated
differential} for \(\kappa \in \KM_+^m(k)\):
\(D^\alpha(f_1;\, x;\, v_\bullet) = D(f_1;\, x;\, v_\bullet^{\times \alpha})\)
for \(m = 1\), and for \(m \geq 2\):

\[
  D^\kappa(f_1, \dots, f_m;\, x;\, v_\bullet) := D(f_m;\, x_{m-1};\, (D^\lambda(f_1, \dots, f_{m-1};\, x;\, v_\bullet)^{\times \kappa(\lambda)})_{\lambda \in \supp(\kappa)}).
\]

For \(H \in \Part_m(k)\) with \(m \geq 2\), set
\(D^H(f_1, \dots, f_m;\, x;\, v_\bullet) := D(f_m;\, x_{m-1};\, (D^L(f_1, \dots, f_{m-1};\, x;\, v_\bullet))_{L \in H})\).

\end{definition}

\begin{proposition}[Iterated infinitesimal Faà di Bruno]

\label{prop:poly-iterated-fdb} Let \(f_1, \dots, f_m\) be polynomial
maps as above, \(z = x_m\), and \(\gamma \in \IN_0^k\).

\begin{enumerate}
\def\labelenumi{\arabic{enumi})}
\tightlist
\item
  \emph{Boolean Faà di Bruno:}
\end{enumerate}

\[
  D(f_m \circ \cdots \circ f_1;\, x;\, v_\bullet) = \sum_{H \in \Part_m(k)} D^H(f_1, \dots, f_m;\, x;\, v_\bullet).
\]

\begin{enumerate}
\def\labelenumi{\arabic{enumi})}
\setcounter{enumi}{1}
\tightlist
\item
  \emph{Binomial Faà di Bruno:}
\end{enumerate}

\[
  D^\gamma(f_m \circ \cdots \circ f_1;\, x;\, v_\bullet) = \sum_{\kappa \in \KM_+^m(k)} \Part_m(\gamma, \kappa)\, D^\kappa(f_1, \dots, f_m;\, x;\, v_\bullet),
\]

where
\(\Part_m(\gamma, \kappa) := \#\set{H \in \Part_m(S(\gamma)) : \nu(H) = \kappa}\)
is the partition grouping coefficient.

\end{proposition}

\begin{proof}

\begin{enumerate}
\def\labelenumi{\arabic{enumi})}
\item
  Evaluate \(f_m \circ \cdots \circ f_1\) on the infinitesimal cube in
  \(A_k\) by composing \(m\) Taylor--Möbius expansions. The nilpotent
  multiplication \(\eps^{T_1} \cdots \eps^{T_p} = 0\) unless the \(T_i\)
  are pairwise disjoint forces partition logic at each level, so only
  higher partitions \(H \in \Part_m(k)\) contribute. Reading off the
  \(\eps^{[k]}\)-component gives the scalar identity.
\item
  Apply part 1 to \(S(\gamma)\) with \(|\gamma|\) directions where
  \(v_i\) is repeated \(\gamma_i\) times. Since \(D^H\) depends only on
  the profile \(\kappa = \nu(H)\), we group the \(m\)-fold partitions of
  \(S(\gamma)\) by profile. The number of \(m\)-fold partitions with
  profile \(\kappa\) is \(\Part_m(\gamma, \kappa)\).
\end{enumerate}

\end{proof}

\begin{proposition}[Partition coefficient recursion]

\label{prop:poly-partition-recursion} \(\Part_m(\gamma, \kappa) = 0\)
unless \(\kappa \vdash \gamma\), and
\(\Part_1(\gamma, \alpha) = \delta_{\gamma, \alpha}\). For \(m \geq 2\)
and \(\kappa \vdash \gamma\), the coefficients satisfy the level
recursion

\[
  \Part_m(\gamma, \kappa) = \frac{\gamma!}{\kappa!} \prod_{\substack{\beta \in \KM_+^{m-1}(S) \\\\ \alpha = \lf(\beta)}}
  \left(
    \frac{1}{\alpha!} \cdot \Part_{m-1}(\alpha, \beta)
  \right)^{\kappa(\beta)}.
\]

The first closed cases are

\[
  \Part_2(\gamma, \kappa) = \frac{\gamma!}{\kappa!} \prod_{\alpha \in \KM_+(S)} (\alpha!)^{-\kappa(\alpha)}, \qquad 
  \Part_3(\gamma, \kappa) = \frac{\gamma!}{\kappa!} \prod_{\beta \in \KM_+^2(S)}
  \left(
    \frac{1}{\beta!} \prod_{\alpha \in \KM_+(S)}  (\alpha!)^{-\beta(\alpha)}
  \right)^{\kappa(\beta)}.
\]

All products here and below are finite, since only factors indexed by
the iterated supports of \(\kappa\) differ from \(1\).

\end{proposition}

\begin{proof}

The leaves of an \(m\)-fold partition \(H \in \Part_m(B)\) of a block
\(B \subseteq S(\gamma)\) partition \(B\), so \(\lf(\nu(H)) = \nu(B)\)
by induction on \(m\). In particular, \(\Part_{m-1}(\lf(\beta), \beta)\)
is the only nonvanishing evaluation of \(\Part_{m-1}(\cdot, \beta)\),
and the profile of an \(m\)-fold partition of \(S(\gamma)\) is a
multi-index partition of \(\gamma\).

For the recursion, an \(m\)-fold partition with profile \(\kappa\) is
built in two independent steps. First, partition \(S(\gamma)\) into
blocks matched with the elements of the multiset \(\kappa\), where a
block matched with \(\beta\) has type \(\lf(\beta)\); there are
\(\gamma! / (\kappa!\, \prod_\beta (\lf(\beta)!)^{\kappa(\beta)})\) such
matched partitions. Second, equip each block matched with \(\beta\) with
an \((m{-}1)\)-fold partition of profile \(\beta\), in
\(\Part_{m-1}(\lf(\beta), \beta)\) ways.

\end{proof}

\subsection{Faà di Bruno deformation}\label{fauxe0-di-bruno-deformation}

The two Taylor algebras \(B_k^\nu\) and \(A_k^\nu\) are fibers of a flat
deformation over \(\ik[t]\). The deformation construction carries over
mutatis mutandis to the iterated and binomial settings via the grid
algebra \(C_k^\nu\); we present the Boolean two-map case here, since it
is the least notation-heavy and already contains the essential
mechanism.

\begin{definition}[Deformation cube algebra]

\label{def:deformation-algebra} For \(\nu \in \IN_0^k\), define:

\[
  C_k^\nu := \ik[t][x_1, \dots, x_k] \,\big/\, \big({\textstyle\prod_{j=0}^{\nu_i}}(x_i - jt) \,:\, 1 \leq i \leq k\big).
\]

For \(\nu = \mathbf{1}\), write \(C_k = C_k^{\mathbf{1}}\); the relation
\(x_i^2 = tx_i\) gives \(x^S x^T = t^{|S \cap T|}\, x^{S \cup T}\). At
\(t = 1\) all coverings contribute; at \(t = 0\) only partitions remain.

\end{definition}

\begin{lemma}[Flatness and fibers]

\label{lem:flatness} \(C_k^\nu\) is a free \(\ik[t]\)-module of rank
\(\prod_i (\nu_i + 1)\). In particular, \(C_k^\nu\) is flat over
\(\ik[t]\), with fibers

\[
  C_k^\nu/(t) \cong A_k^\nu, \qquad C_k^\nu/(t-1) \cong B_k^\nu.
\]

\end{lemma}

\begin{proof}

The relations \(\prod_{j=0}^{\nu_i}(x_i - jt) = 0\) are monic of degree
\(\nu_i + 1\) in \(x_i\), so \(C_k^\nu\) is free over \(\ik[t]\) with
monomial basis \(\set{x^\alpha : 0 \leq \alpha \leq \nu}\). The fiber
identifications follow by substituting \(t = 0\) and \(t = 1\).

\end{proof}

\begin{proposition}[Taylor--Möbius duality in $C_k$]

\label{prop:deformation-taylor} Let \(p \in \ik[y_1, \dots, y_d]\),
\(x \in \ik^d\), and \(v_1, \dots, v_k \in \ik^d\). Expand
\(p(x + \sum_i v_i x_i) = \sum_{S \subseteq [k]} c_S(t)\, x^S\) in
\(C_k\). Then

\[
  c_S(t) = \frac{1}{t^{|S|}}\, \Delta(p; x; (tv_i)_{i \in S}), \qquad c_S(0) = D(p; x; v_S).
\]

The coefficients \(c_S(t)\) lie in \(\ik[t]\); in particular,
\(\Delta(p; x; tv_1, \dots, tv_k)\) is divisible by \(t^k\) in
\(\ik[t]\).

\end{proposition}

\begin{proof}

Write \(p(x + \sum_i v_i x_i) = \sum_S c_S\, x^S\) in \(C_k\) (free over
\(\ik[t]\) with basis \(x^S\)). In the generic fiber \(C_k[t^{-1}]\),
set \(\delta_i = t^{-1} x_i\), so \(\delta_i^2 = \delta_i\). The branch
maps \(\rho_T: \delta_i \mapsto 1_{i \in T}\) give
\(\rho_T(p(x + \sum_i v_i x_i)) = p(x + \sum_{i \in T} tv_i)\). Since
\(\rho_T(x^S) = t^{|S|} \cdot 1_{S \subseteq T}\), Boolean Möbius
inversion \ref{prop:moebius-inversion} gives
\(t^{|S|} c_S(t) = \sum_{T \subseteq S} (-1)^{|S|-|T|}\, p(x + \sum_{i \in T} tv_i) = \Delta(p; x; (tv_i)_{i \in S})\),
hence \(c_S(t) = t^{-|S|}\, \Delta(p; x; (tv_i)_{i \in S})\). Since
\(c_S \in \ik[t]\) (freeness), the divisibility follows, and
\(c_S(0) = D(p; x; v_S)\) by definition.

\end{proof}

\begin{lemma}[Multilinear part of the mixed difference]

\label{lem:multilinear-part} For a polynomial \(p\), a point \(y\), and
\(w_1, \dots, w_r \in \ik^d\):

\[
  \Delta(p;\, y;\, \eps_1 w_1, \dots, \eps_r w_r) = D(p;\, y;\, w_1, \dots, w_r)\, \eps^{[r]} \qquad \text{in } A_r.
\]

In particular, the multilinear part of \(\Delta(p; y; w_\bullet)\) in
\((w_1, \dots, w_r)\) equals \(D(p; y; w_\bullet)\).

\end{lemma}

\begin{proof}

This is the \(A_r\) Möbius display of \ref{prop:poly-taylor-duality}
applied at \(y\) with directions \(w_\bullet\): the left side is the
alternating sum
\(\sum_{T \subseteq [r]} (-1)^{r-|T|}\, p(y + \sum_{j \in T} \eps_j w_j)\),
and the \(\eps^{[r]}\)-coefficient of \(p(y + \sum_j \eps_j w_j)\) is
\(D(p; y; w_\bullet)\) by definition. Substituting
\(w_j \mapsto \eps_j w_j\) retains exactly the multilinear monomials of
\(\Delta(p; y; w_\bullet)\), since \(\eps_j^2 = 0\).

\end{proof}

\begin{theorem}[Faà di Bruno deformation]

\label{thm:fdb-deformation} Let \(q: \ik^e \to \ik^d\) and
\(p: \ik^d \to \ik\) be polynomials, \(x \in \ik^e\),
\(v_1, \dots, v_k \in \ik^e\), and \(y = q(x)\). Write \(c_S(t)\) for
the \(x^S\)-coefficient of \((p \circ q)(x + \sum_i v_i x_i) \in C_k\).
Then

\[
  c_S(t) = \frac{1}{t^{|S|}} \sum_{\substack{H \subseteq \KP_+(S) \\\\ \bigcup H = S}} \Delta(p;\, y;\, (\Delta(q;\, x;\, (tv_i)_{i \in T}))_{T \in H}), \qquad c_S(0) = \sum_{\pi \in \Part(S)} D(p;\, y;\, (D(q;\, x;\, v_B))_{B \in \pi}).
\]

Each covering summand is divisible by \(t^{\mathrm{wt}(H)}\) (since
\(\Delta(q; x; (tv_i)_{i \in T})\) is divisible by \(t^{|T|}\)), so
\(c_S \in \ik[t]\). At \(t = 0\), only the coverings with
\(\mathrm{wt}(H) = |S|\), namely the partitions, contribute.

\end{theorem}

\begin{proof}

By \ref{prop:deformation-taylor} applied to the composite (extending
scalars to \(\ik[t]\)),
\(c_S(t) = t^{-|S|}\, \Delta(p \circ q;\, x;\, (tv_i)_{i \in S})\). The
discrete covering formula \ref{thm:dfdb}, applied over \(\ik[t]\) with
increments \((tv_i)_{i \in S} \in \ik[t]^e\), expands this as
\(t^{-|S|} \sum_{H \in \Cov(S)} \Delta(p;\, y;\, (\Delta(q;\, x;\, (tv_i)_{i \in T}))_{T \in H})\).
Each inner increment is divisible by \(t^{|T|}\) by
\ref{prop:deformation-taylor}; write
\(\Delta(q; x; (tv)_T) = t^{|T|} \tilde{q}_T(t)\) with
\(\tilde{q}_T \in \ik[t]^d\) and \(\tilde{q}_T(0) = D(q; x; v_T)\). As a
polynomial in \((w_1, \dots, w_r) \in (\ik^d)^r\), the mixed difference
\(\Delta(p; y; w_1, \dots, w_r)\) vanishes on every hyperplane
\(w_j = 0\), so each of its monomials contains at least one factor from
each \(w_j\); substituting \(w_j = t^{|T_j|} \tilde{q}_{T_j}(t)\)
therefore extracts \(\prod_{T \in H} t^{|T|} = t^{\mathrm{wt}(H)}\).
Hence \(c_S(t) \in \ik[t]\). At \(t = 0\), terms with
\(\mathrm{wt}(H) > |S|\) vanish; the remaining terms have
\(\mathrm{wt}(H) = |S|\), i.e.~\(H\) is a partition
\ref{lem:weight-bound}. For a partition \(\pi\) with blocks
\(B_1, \dots, B_r\), the \(\pi\)-summand is
\(\Delta(p;\, y;\, (t^{|B_j|} \tilde{q}_{B_j}(t))_j)\); by the
multilinear-part lemma \ref{lem:multilinear-part}, the multilinear
monomials contribute
\(t^{\sum |B_j|} D(p;\, y;\, (\tilde{q}_{B_j}(t))_j) = t^{|S|} D(p;\, y;\, (\tilde{q}_{B_j}(t))_j)\),
while all higher monomials contribute above order \(|S|\). Dividing by
\(t^{|S|}\) and setting \(t = 0\) gives
\(D(p;\, y;\, (D(q;\, x;\, v_B))_{B \in \pi})\), since
\(\tilde{q}_B(0) = D(q; x; v_B)\) \ref{prop:deformation-taylor}.

\end{proof}

\section{Fréchet Möbius calculus}\label{fruxe9chet-muxf6bius-calculus}

The only analytic input needed to pass from polynomial maps to \(C^n\)
maps between Banach spaces is that a \(C^n\) composite is computed by
its Taylor polynomials up to order \(n\); this is
\ref{thm:taylor-composition}, proved in \citep{Hartmann2025FDB}.
Everything else is the polynomial calculus of the previous section
applied to jets.

The polynomial calculus of the previous section extends verbatim to
polynomial maps between \(\ik\)-modules, and in particular to continuous
polynomial maps between Banach spaces: the proofs use only base change
along \(\ik \to A_k^\nu\), which extends canonically to symmetric
multilinear maps, and the branch maps \(\rho_\beta\) act on the
\(\eps\)-variables only, never on coordinates of the target.

Let \(X_0, \dots, X_m\) be Banach spaces and let
\(f_r: X_{r-1} \to X_r\) be \(C^n\) near the relevant basepoints. Put
\(x_0 = x\) and \(x_r = f_r(x_{r-1})\). We write \(T^n(f_r; x_{r-1})\)
for the order-\(n\) Taylor polynomial of \(f_r\) at \(x_{r-1}\),
regarded as a polynomial map on \(X_{r-1}\). For a \(C^n\) map \(f\),
the mixed directional derivative
\(D^\alpha(f; x; v_\bullet) := D^{|\alpha|}f(x)[v_1^{\otimes \alpha_1}, \dots, v_k^{\otimes \alpha_k}]\)
agrees with \(D^\alpha(T^n(f; x); x; v_\bullet)\) for
\(|\alpha| \leq n\); the iterated differential \(D^\kappa\) for
\(\kappa \in \KM_+^m(k)\) is defined by the same recursion as
\ref{def:poly-iterated-diff}, applied to the Taylor jets. Every
\(\kappa\) with \(\Part_m(\gamma, \kappa) \neq 0\) satisfies
\(\kappa \vdash \gamma\), so the sum in \ref{thm:frechet-fdb} is finite
and only derivatives of order \(\leq |\gamma| \leq n\) occur.

\begin{proposition}[Taylor composition for Fréchet maps]

\label{thm:taylor-composition} For \(C^n\) maps between Banach spaces,

\[
  T^n(f_m \circ \cdots \circ f_1;\, x) = \pi_{\leq n}(T^n(f_m;\, x_{m-1}) \circ \cdots \circ T^n(f_1;\, x_0)),
\]

where \(\pi_{\leq n}\) denotes truncation to total degree at most \(n\).

\end{proposition}

\begin{proof}

Write \(f_r = P_r + R_r\) where \(P_r = T^n(f_r; x_{r-1})\) is the
Taylor polynomial and \(R_r\) is the Peano residual with
\(R_r(h) = o(\|h\|^n)\). Substituting into the composite and expanding,
every term containing at least one factor of \(R_r\) is \(o(\|h\|^n)\),
so the Taylor jet of the composite up to order \(n\) agrees with the
truncated composite of the Taylor polynomials. Details are given in
\citep{Hartmann2025FDB}.

\end{proof}

Thus the Taylor--Möbius duality and Faà di Bruno formulas of the
previous section, which are identities for polynomial maps proved via
the infinitesimal grid in \(A_k^\nu\), apply directly to the Taylor jets
of \(C^n\) maps.

\begin{theorem}[Fréchet iterated Faà di Bruno / recursive Constantine--Savits]

\label{thm:frechet-fdb} Let \(X_0, \dots, X_m\) be Banach spaces, let
\(f_r: X_{r-1} \to X_r\) be \(C^n\) near \(x_{r-1}\), and put
\(x_0 = x\), \(x_r = f_r(x_{r-1})\), \(z = x_m\). Fix directions
\(v_1, \dots, v_k \in X_0\) and let \(\gamma \in \IN_0^k\) with
\(1 \leq |\gamma| \leq n\).

\begin{enumerate}
\def\labelenumi{\arabic{enumi})}
\tightlist
\item
  \emph{Taylor composition:}
\end{enumerate}

\[
  (f_m \circ \cdots \circ f_1)(x + {\textstyle\sum_i} t_i v_i) = z + \sum_{\substack{0 < |\gamma| \leq n \\\\ \kappa \in \KM_+^m(k)}} \frac{t^\gamma}{\gamma!}\, \Part_m(\gamma, \kappa)\, D^\kappa(f_1, \dots, f_m;\, x;\, v_\bullet) + o(|t|^n).
\]

\begin{enumerate}
\def\labelenumi{\arabic{enumi})}
\setcounter{enumi}{1}
\tightlist
\item
  \emph{Faà di Bruno:}
\end{enumerate}

\[
  D^\gamma(f_m \circ \cdots \circ f_1;\, x;\, v_\bullet) = \sum_{\kappa \in \KM_+^m(k)} \Part_m(\gamma, \kappa)\, D^\kappa(f_1, \dots, f_m;\, x;\, v_\bullet).
\]

Here \(D^\kappa\) is the recursive iterated differential
\ref{def:poly-iterated-diff} and \(\Part_m(\gamma, \kappa)\) is the
partition grouping coefficient \ref{prop:poly-partition-recursion}.

\end{theorem}

\begin{proof}

By Taylor composition \ref{thm:taylor-composition},
\(F(x + h) = \pi_{\leq n}(P_m \circ \cdots \circ P_1)(x + h) + o(\|h\|^n)\),
where \(P_r = T^n(f_r; x_{r-1})\) are polynomial maps between Banach
spaces. The polynomial Faà di Bruno \ref{prop:poly-iterated-fdb}, which
applies to polynomial maps between modules, gives 2) for the composite
of the Taylor jets. For 1), restrict to \(h = \sum_i t_i v_i\) with
\(\|h\| \leq C|t|\), so \(o(\|h\|^n) = o(|t|^n)\); expanding the
truncated polynomial composite in \(t\) by binomial Taylor--Möbius
duality \ref{prop:binomial-taylor-duality} and inserting 2) gives the
stated formula.

\end{proof}

\begin{corollary}[Multivariate Taylor formula, \citep{LangRFA}]

\label{cor:multivariate-taylor} For a single \(C^n\) map \(f: X \to Y\)
between Banach spaces, \(x \in X\), and \(v_1, \dots, v_k \in X\):

\[
  f(x + {\textstyle\sum_i} t_i v_i) = f(x) + \sum_{1 \leq |\alpha| \leq n} \frac{t^\alpha}{\alpha!}\, D^\alpha(f;\, x;\, v_\bullet) + o(|t|^n).
\]

\end{corollary}

\begin{corollary}[Constantine--Savits, \citep{CS1996}]

\label{cor:cs} For two \(C^n\) maps \(f_1: X_0 \to X_1\),
\(f_2: X_1 \to X_2\) between Banach spaces, \(x \in X_0\),
\(x_1 = f_1(x)\), and \(\gamma \in \IN_0^k\) with \(|\gamma| \leq n\):

\[
  D(f_2 \circ f_1;\, x;\, v_\bullet^{\times \gamma}) = \sum_{\kappa \vdash \gamma} \frac{\gamma!}{\kappa!\, \prod_\alpha (\alpha!)^{\kappa(\alpha)}}\, D(f_2;\, x_1;\, (D(f_1;\, x;\, v_\bullet^{\times \alpha}))_\alpha^{\times \kappa(\alpha)}).
\]

\end{corollary}

\begin{corollary}[Univariate Faà di Bruno]

\label{cor:univariate-fdb} For \(C^n\) functions \(f, g: \IR \to \IR\),
the Constantine--Savits formula \ref{cor:cs} with \(k = 1\) and
\(\gamma = n\) gives:

\[
  (f \circ g)^{(n)}(x) = \sum_{\substack{k_1 + 2k_2 + \cdots + nk_n = n \\\\ k_j \geq 0}} \frac{n!}{k_1!\, k_2! \cdots k_n!}\, f^{(k_1 + \cdots + k_n)}(g(x)) \prod_{j=1}^n \left(\frac{g^{(j)}(x)}{j!}\right)^{k_j}.
\]

This is the classical Faà di Bruno formula in Bell polynomial form.

\end{corollary}

\section{Applications}\label{applications}

Each section exercises a different face of the covering calculus: a
combinatorial self-count, a degree bound for polynomial maps between
abelian groups, a composition law for Newton and Mahler coefficients,
and a fast algorithm for composing discrete jets. The final section on
product rules is a complement rather than an application: the same
Möbius mechanism run for products instead of composites.

\subsection{Self-enumeration}\label{self-enumeration}

\begin{proposition}[Covering counts]

\label{prop:cover-count} Let \(g: \IN_0 \to \IZ\), \(g(x) = 2^x - 1\)
and \(f: \IN_0 \to \IZ\), \(f(y) = 2^y\). Then

\[
  |\Cov(k)| = \sum_{j=0}^{k} (-1)^{k-j} \binom{k}{j}\, 2^{2^j - 1}.
\]

\end{proposition}

\begin{proof}

Every increment equals \(1\): \(\Delta^j g(0) = 1\) for \(j \geq 1\) and
\(\Delta^p f(0) = 1\) for \(p \geq 1\) at \(y = g(0) = 0\). The covering
formula \ref{thm:dfdb} gives
\(\Delta^k(f \circ g)(0) = \sum_{H \in \Cov(k)} 1 = |\Cov(k)|\). The
Newton expansion of the left side gives the stated identity.

\end{proof}

The covering formula counts its own terms, the exact discrete parallel
of the classical fact that \(e^{e^x - 1}\) counts partitions (Bell
numbers): exponential maps are the group-like elements where every
increment is \(1\), and each formula enumerates its own index set.

\[
\begin{array}{l|rrrrr}
 & k=1 & k=2 & k=3 & k=4 & k=5 \\
\hline
|\Cov(k)| & 1 & 5 & 109 & 32{,}297 & 2{,}147{,}321{,}017 \\
|\Part(k)| = B_k & 1 & 2 & 5 & 15 & 52
\end{array}
\]

\subsection{Degree bound for polynomial
maps}\label{degree-bound-for-polynomial-maps}

A map \(g: X \to Y\) between abelian groups is \emph{polynomial of
degree \(\leq d\)} if \(\Delta^{d+1} g \equiv 0\) (see \citep{MO1934},
\citep{Leibman2002}).

\begin{corollary}

\label{cor:degree-bound} If \(g: X \to Y\) is polynomial of degree
\(\leq d\) and \(f: Y \to Z\) is polynomial of degree \(\leq e\), then
\(f \circ g\) is polynomial of degree \(\leq ed\).

\end{corollary}

\begin{proof}

In the covering formula \ref{thm:dfdb}, the \(H\)-term involves
\(\Delta^{|H|}f\) (vanishing for \(|H| > e\)) applied to directions
\(\Delta^{|T|}g\) (vanishing for \(|T| > d\)). A covering with
\(|H| \leq e\) blocks of size \(|T| \leq d\) satisfies
\(k \leq \sum_{T \in H} |T| \leq ed\).

\end{proof}

\subsection{Newton--Mahler composition}\label{newtonmahler-composition}

For \(g: \IN_0 \to \IZ\) with Newton coefficients
\(b_j = \Delta^j g(0)\) and any \(f: \IZ \to Y\), the binomial covering
formula \ref{thm:iterated-dfdb} with \(\gamma = n\) and unit increments
gives the Newton coefficients of \(f \circ g\) in terms of the Newton
coefficients of \(g\) and finite differences of \(f\) at \(g(0)\):

\[
  \Delta^n(f \circ g)(0) = \sum_{\kappa} \Cov_2(n, \kappa)\, \Delta(f;\, g(0);\, (b_j)_j^{\times \kappa(j)}),
\]

where the sum runs over all \(\kappa\) with \(\Cov_2(n, \kappa) \neq 0\)
(finitely many). Since Mahler's theorem identifies continuous maps
\(\IZ_p \to \IZ_p\) with convergent Newton series (\(b_n \to 0\)
\(p\)-adically), the same formula computes Mahler coefficients of
composites of continuous \(p\)-adic maps.

\subsection{Discrete jet composition}\label{discrete-jet-composition}

The fixed-basepoint property makes discrete jets composable without
re-evaluation. A \emph{discrete \(k\)-jet} of \(g\) at \(x\) in
directions \(u_\bullet\) is the table
\((\Delta(g; x; u_T))_{T \subseteq [k]}\) of \(2^k\) forward
differences. Given this table, the discrete jet of \(f \circ g\) is
computed in three steps:

\begin{enumerate}
\def\labelenumi{\arabic{enumi}.}
\tightlist
\item
  \emph{Zeta (sum up):} recover the \(2^k\) vertex values
  \(g(x + \sum_{i \in S} u_i)\) from the differences by discrete Taylor
  \ref{prop:taylor-duality}.
\item
  \emph{Evaluate:} apply \(f\) pointwise to get
  \((f \circ g)(x + \sum_{i \in S} u_i)\) for all \(S \subseteq [k]\).
\item
  \emph{Möbius (alternate):} extract the differences
  \(\Delta(f \circ g; x; u_S)\) by the alternating sieve.
\end{enumerate}

Steps 1 and 3 are the fast zeta and Möbius transforms on
\(\set{0,1}^k\), each costing \(O(k \cdot 2^k)\) operations. The total
cost is \(O(k \cdot 2^k)\), polynomial in the jet size \(2^k\). The
covering formula is what this algorithm produces when expanded
algebraically; the zeta/Möbius factorization is the efficient
implementation.

\subsection{Product rules}\label{product-rules}

\begin{theorem}[Discrete product rule]

\label{thm:product-rule} Let \(X\) be an abelian group, \(A\) an
associative algebra, and \(f_1,\dots,f_r: X \to A\). For \(x \in X\) and
\(u_1,\dots,u_n \in X\):

\begin{enumerate}
\def\labelenumi{\arabic{enumi})}
\tightlist
\item
  \emph{Boolean product rule:}
\end{enumerate}

\[
  \Delta(f_1 \cdots f_r; x; u_1,\dots,u_n)
  =
  \sum_{\substack{J_1,\dots,J_r \subseteq [n] \\\\ J_1 \cup \cdots \cup J_r = [n]}}
  \Delta(f_1; x; u_{J_1}) \cdots \Delta(f_r; x; u_{J_r}).
\]

\begin{enumerate}
\def\labelenumi{\arabic{enumi})}
\setcounter{enumi}{1}
\tightlist
\item
  \emph{Boolean Taylor product rule:}
\end{enumerate}

\[
  (f_1 \cdots f_r)(x + {\textstyle\sum_{i=1}^n} u_i)
  =
  \sum_{J_1, \dots, J_r \subseteq [n]}
  \Delta(f_1; x; u_{J_1}) \cdots \Delta(f_r; x; u_{J_r}).
\]

The sum in 1) runs over all \emph{ordered coverings} \((J_1,\dots,J_r)\)
of \([n]\): the \(J_a\) may overlap, and empty \(J_a\) are allowed. Part
2) is the product of Taylor expansions. The binomial versions follow by
applying 1,2) to \(S(\gamma)\).

\end{theorem}

\begin{proof}

By discrete Taylor \ref{prop:taylor-duality},
\(f_a(x + \sum_{i \in S} u_i) = \sum_{J_a \subseteq S} \Delta(f_a; x; u_{J_a})\).
Substituting into the alternating sum for \(f_1 \cdots f_r\) and
expanding the product gives:

\[
  \Delta(f_1 \cdots f_r; x; u_1, \dots, u_n)
  =
  \sum_{\substack{J_1, \dots, J_r \subseteq [n] \\\\ S \supseteq J_1 \cup \cdots \cup J_r}}
  (-1)^{n-|S|}\,
  \Delta(f_1; x; u_{J_1}) \cdots \Delta(f_r; x; u_{J_r}).
\]

Exchanging the order of summation, a fixed tuple
\((J_1, \dots, J_r) \in \KP([n])^r\) appears in the \(S\)-summand
exactly when \(J_1 \cup \cdots \cup J_r \subseteq S\). Setting
\(M = [n] \setminus (J_1 \cup \cdots \cup J_r)\), the sign sum is
\(\sum_{R \subseteq M} (-1)^{|M|-|R|} = (1-1)^{|M|}\): this is \(1\) if
\(J_1 \cup \cdots \cup J_r = [n]\) and \(0\) otherwise.

\end{proof}

\begin{theorem}[Fréchet product rule]

\label{thm:frechet-product-rule} Let \(X, Y\) be Banach spaces, \(A\) a
Banach algebra, and \(f_1, \dots, f_r: X \to A\) be \(C^n\) near \(x\),
with \(n \geq k\). For \(v_1, \dots, v_k \in X\):

\begin{enumerate}
\def\labelenumi{\arabic{enumi})}
\tightlist
\item
  \emph{Boolean product rule:}
\end{enumerate}

\[
  D(f_1 \cdots f_r; x; v_{[k]})
  =
  \sum_{\substack{J_1 \sqcup \cdots \sqcup J_r = [k]}}
  D(f_1; x; v_{J_1}) \cdots D(f_r; x; v_{J_r}).
\]

\begin{enumerate}
\def\labelenumi{\arabic{enumi})}
\setcounter{enumi}{1}
\tightlist
\item
  \emph{Boolean Taylor product rule:} the square-free part of the
  order-\(n\) Taylor polynomial of
  \(t \mapsto (f_1 \cdots f_r)(x + \sum_{i=1}^k t_i v_i)\) is
\end{enumerate}

\[
  \sum_{J_1 \sqcup \cdots \sqcup J_r \subseteq [k]}
  D(f_1; x; v_{J_1}) \cdots D(f_r; x; v_{J_r})\, t^{J_1 \sqcup \cdots \sqcup J_r}.
\]

\begin{enumerate}
\def\labelenumi{\arabic{enumi})}
\setcounter{enumi}{2}
\tightlist
\item
  \emph{Binomial product rule:} For \(\gamma \in \IN_0^k\) with
  \(|\gamma| \leq n\):
\end{enumerate}

\[
  D(f_1 \cdots f_r; x; v_\bullet^{\times \gamma})
  =
  \sum_{\substack{\alpha_1 + \cdots + \alpha_r = \gamma}}
  \frac{\gamma!}{\alpha_1! \cdots \alpha_r!}\,
  D(f_1; x; v_\bullet^{\times \alpha_1}) \cdots D(f_r; x; v_\bullet^{\times \alpha_r}).
\]

\begin{enumerate}
\def\labelenumi{\arabic{enumi})}
\setcounter{enumi}{3}
\tightlist
\item
  \emph{Binomial Taylor product rule:}
\end{enumerate}

\[
  (f_1 \cdots f_r)(x + {\textstyle\sum_i} t_i v_i)
  =
  \sum_{\substack{\alpha_1, \dots, \alpha_r \geq 0 \\\\ |\alpha_1 + \cdots + \alpha_r| \leq n}}
  \frac{t^{\alpha_1 + \cdots + \alpha_r}}{\alpha_1! \cdots \alpha_r!}\,
  D(f_1; x; v_\bullet^{\times \alpha_1}) \cdots D(f_r; x; v_\bullet^{\times \alpha_r}) + o(|t|^n).
\]

The sum in 1) runs over all \emph{ordered partitions}
\((J_1, \dots, J_r)\) of \([k]\): the \(J_a\) are pairwise disjoint and
empty \(J_a\) are allowed. This is the discrete product rule with
coverings replaced by partitions.

\end{theorem}

\begin{proof}

By Taylor composition \ref{thm:taylor-composition}, the coefficients of
the product \(f_1 \cdots f_r\) up to order \(n\) agree with those of the
product of Taylor polynomials. For polynomials, the result follows from
the discrete product rule applied in the Taylor algebra \(A_k\): since
\(\eps^{J_1} \cdots \eps^{J_r} = 0\) unless the \(J_a\) are pairwise
disjoint, the ordered coverings of the discrete rule reduce to ordered
partitions.

\end{proof}

\section{Acknowledgements}\label{acknowledgements}

We thank Duarte and Torres for their wonderfully dense paper
\citep{DT2008}, which inspired this work, and for sharing the source
code accompanying their formula.

\bibliographystyle{alpha-local}
\bibliography{refs}

\end{document}